\documentclass[a4paper,12pt,reqno]{amsart}
\usepackage{graphicx} 
\usepackage{mathabx}
\usepackage{float} 
\usepackage{subfigure} 
\usepackage{tikz}
\usetikzlibrary{calc}
\usetikzlibrary{patterns}
\usepackage{amsfonts}
\usepackage{amssymb}
\usepackage{amsmath}
\usepackage{amsthm,color}
\usepackage{pifont}
\numberwithin{equation}{section}

\allowdisplaybreaks

\renewcommand{\Re}{\operatorname{Re}}
\renewcommand{\Im}{\operatorname{Im}}

\newtheorem{thm}{Theorem}[section]
\newtheorem{prop}[thm]{Proposition}
\newtheorem{lem}[thm]{Lemma}
\newtheorem{cor}[thm]{Corollary}

\newtheorem{thmy}{Theorem}
\newenvironment{thmx}{\begin{thmy}}{\end{thmy}}

\renewcommand{\theequation}
{\thesection.\arabic{equation}}

\DeclareMathOperator\supp{supp}

\def\mr{\mathbb{R}}

\def\Lp{L^{p}}
\def\L1{L^{1}}

\def\B0{B_{0}}

\begin{document}

\title[variational inequalities for generalized spherical means]{variational inequalities for generalized spherical means}
\author{Wenjuan Li, Dongyong Yang and Feng Zhang}

\address{Wenjuan Li, School of Mathematical and Statistics\\
	Northwestern Polytechnical  University\\
	Xi'an 710129, People's Republic of China}

\email{liwj@nwpu.edu.cn}

\address{Dongyong Yang, School of Mathematical Sciences\\
	Xiamen University\\
	Xiamen 361005, People's Republic of China}

\email{dyyang@xmu.edu.cn}

\address{Feng Zhang, School of Mathematical Sciences\\
	Xiamen University\\
	Xiamen 361005, People's Republic of China}

\email{fengzhang@stu.xmu.edu.cn}

\makeatletter
\@namedef{subjclassname@2020}{\textup{2020} Mathematics Subject Classification}
\makeatother

\date{\today}

\begin{abstract}
	In this paper, we establish the $L^{p}(\mathbb{R}^{d})$-boundedness of the variation operator and the $\delta$-jump operator for generalized spherical means, and we also show the necessary conditions for the $L^{p}(\mathbb{R}^{d})$-boundedness of these operators. These results are almost optimal when $d=2$.  
\end{abstract}

\maketitle



\section{Introduction and statement of main results\label{s1}}
Let $q\in [1,\infty)$ and $\vec{a}:=\{a_{t}\}_{t>0}$ be a family of complex numbers, the $q$-variation of $\vec{a}$ is defined by
\begin{align}\label{eq-var-def}
	\|\vec{a}\|_{v_{q}}:=\sup_{L\in \mathbb{N}}\sup_{t_{1}<\cdots<t_{L}}\Big(\sum_{i=1}^{L-1}|a_{t_{i+1}}-a_{t_{i}}|^{q}\Big)^{\frac{1}{q}},
\end{align}
where the supremum is taken over all $L\in \mathbb{N}$ and all sequences $\{t_{i}:0<t_{1}<\cdots<t_{L}<\infty\}$. For $q=\infty$, 
\begin{align*}
	\|\vec{a}\|_{v_{\infty}}:=\sup_{L\in \mathbb{N}}\sup_{1\leq i\leq L-1}|a_{t_{i+1}}-a_{t_{i}}|.
\end{align*}
Denote by $v_{q}$ the space of all functions on $(0,\infty)$ with finite $q$-variation norm $\|\cdot\|_{v_{q}}$ as in \eqref{eq-var-def}. It is a Banach space modulo constant functions. 
If 
$$\|\vec{a}\|_{v_{q}}<\infty$$
for some $q\in [1,\infty)$, then the limits $\lim_{t\rightarrow0}a_{t}$ and $\lim_{t\rightarrow\infty}a_{t}$ exist. Let $\{T_{t}\}_{t>0}$ be a family of bounded operators on $L^{p}(\mathbb{R}^{d})$ for some $p\in (1,\infty)$ and 
\begin{align*}
	V_{q}(\mathcal{T}f)(x):=\|\{T_{t}f(x)\}_{t>0}\|_{v_{q}}.
\end{align*}
If $V_{q}(\mathcal{T})$ is bounded on $L^{p}(\mathbb{R}^{d})$, then for any $f\in L^{p}(\mathbb{R}^{d})$, the limits $\lim_{t\rightarrow0}T_{t}f(x)$ and $\lim_{t\rightarrow\infty}T_{t}f(x)$ exist for almost everywhere $x\in \mathbb{R}^{d}$. So, the variational inequality is an important tool to obtain almost everywhere convergence of a family of operator $\{T_{t}\}_{t>0}$ without showing the convergence for a dense class previously, which is challenging in the ergodic theory and probability. Variational inequalities have received a lot of attention in probability, ergodic theory and harmonic analysis. L\'epingle \cite{MR420837} obtained $q$-variational inequalities for martingales with $q>2$. These $q$-variational inequalities may fail for $q\leq 2$; see \cite{Jones04TransAMS,Qian1998AnnProb}. Alternative proofs of L\'epingle's result are provided by Pisier and Xu \cite{Pisier88Prob} and by Bourgain \cite{MR1019960}. Bourgain \cite{MR1019960} also established a variational inequality for the ergodic averages of a dynamic system to obtain Birkhoff's pointwise ergodic theorem, whose work initiated a systematic study of variational inequalities in probability, ergodic theory and harmonic analysis, see also \cite{Beltran2022MathAnn,Friz20AnnProb,Guo2020Anal,Jones08Trans,Mir17Inv,Mirek20AnalPDE,Mir20Adv,Mir20MathAnn,Obe12JEMS} and their references.

The main goal of this paper is to consider the variation operator for generalized spherical means. Stein \cite{Stein1976Acad} considered the generalized spherical maximal means $$M^{\alpha}f(x):=\sup_{t>0}|A_t^{\alpha} f(x)|,$$ 
where $A_t^{\alpha} f(x)$ is given by
\begin{align*}
	A_t^{\alpha} f(x):=\frac{1}{\Gamma(\alpha)} \int_{|y| \leq 1}\left(1-|y|^2\right)^{\alpha-1} f(x-t y)dy.
\end{align*}
The generalized spherical means are defined a priori only for $\Re \alpha>0$. However, by a direct calculation (see \cite[p.~171]{Stein1973book} and \cite[Appendix A]{miao2017Proc}), the Fourier transform of $A_t^{\alpha}f$ is given by
\begin{align}\label{eq-average-def}
	\widehat{A_{t}^{\alpha}f}(\xi)=\widehat{f}(\xi)\pi^{-\alpha+1}|t\xi|^{-d/2-\alpha+1}J_{d/2+\alpha-1}(2\pi|t\xi|)=:\widehat{f}(\xi)m^{\alpha}(t\xi),
\end{align}
where $J_{\beta}$ denotes the Bessel function of order $\beta$. Recall that for $r>0$ and $\beta\in \mathbb{C}$, the Bessel function is given by
\begin{align*}
	J_\beta(r):=\sum_{j=0}^{\infty} \frac{(-1)^j}{j !} \frac{1}{\Gamma(j+\beta+1)}\left(\frac{r}{2}\right)^{2 j+\beta},
\end{align*}
we refer the readers to \cite[Chapter \uppercase\expandafter{\romannumeral2}]{bookWatson} for more details.
Thus, the definition of $A_t^{\alpha} f$ can be extended to $\alpha\in \mathbb{C}$ by \eqref{eq-average-def}. Obviously, the averages over Euclidean balls can be recovered by taking $\alpha=1$ and the spherical means can be recovered by taking $\alpha=0$. In \cite{Stein1976Acad}, Stein showed that $M^{\alpha}$ is bounded on $\Lp(\mathbb{R}^{d})$ if 
\begin{align}\label{eq-range-p1}
	1<p\leq 2 \text{ and } \Re \alpha>1-d+\frac{d}{p}
\end{align}
or
\begin{align}\label{eq-range-p2}
	2\leq p\leq \infty  \text{ and } \Re\alpha>\frac{2-d}{p}.
\end{align}
The inequalities in \eqref{eq-range-p1} are sharp; see \cite[p.~519]{Stein93}. This fact implies that $M^{0}$ is bounded on $\Lp(\mathbb{R}^{d})$ whenever $p>d/(d-1)$ and $d\geq 3$. Bourgain \cite{Bou86JAM} established the $\Lp(\mathbb{R}^{2})$-boundedness of $M^{0}$ for $p>2$. In 1992, Mockenhaupt, Seeger and Sogge \cite{Mockenhaupt1992Ann} provided an alternative proof of Bourgain's result using a local smoothing estimate, they also improved the range of $p$ in \eqref{eq-range-p2} when $d=2$. Based on the Bourgain-Demeter $\ell^{2}$ decoupling theorem \cite{Bou15AnnMath}, Miao, Yang and Zheng further improved \eqref{eq-range-p2} to
\begin{align*}
	2\leq p\leq \frac{2(d+1)}{d-1} \text{ and } \Re \alpha>\frac{1-d}{4}+\frac{3-d}{2p}
\end{align*}
or
\begin{align*}
	\frac{2(d+1)}{d-1}\leq p \leq \infty \text{ and } \Re \alpha>\frac{1-d}{p}.
\end{align*}
Nowak, Roncal and Szarek \cite{Nowak23CPAA} showed the optimal conditions for the generalized spherical maximal means on radial functions.
Recently, Liu, Shen, Song and Yan \cite{Liu2023arXiv} obtained the essentially sharp result in dimension $2$, that is, $M^{\alpha}$ is bounded on $L^{p}(\mathbb{R}^{2})$ if
\begin{align*}
	2\leq p\leq 4 \text{ and } \Re\alpha>\frac{1}{p}-\frac{1}{2}
\end{align*}
or
\begin{align*}
	4\leq p\leq \infty \text{ and } \Re\alpha>-\frac{1}{p}.
\end{align*}

In this paper, we aim to establish the $L^{p}(\mathbb{R}^{d})$-boundedness of the variation operator for the generalized spherical means $A_{t}^{\alpha}$ defined by \eqref{eq-average-def}. 

For $q>2$, the variation operator $V_{q}(\mathcal{A}^{\alpha})$ for generalized spherical average is given by 
$$V_{q}(\mathcal{A}^{\alpha}f)(x):=\|\mathcal{A}^{\alpha}f(x)\|_{v_{q}},$$
where $\mathcal{A}^{\alpha}f(x):=\{A_{t}^{\alpha}f(x)\}_{t>0}$. From the definition of the variation norm, it follows that
\begin{align*}
	M^{\alpha}f(x)\leq V_{\infty}(\mathcal{A}^{\alpha}f)(x)+|A_{t_{0}}^{\alpha}f(x)|\leq 2M^{\alpha}f(x)+|A_{t_{0}}^{\alpha}f(x)|
\end{align*}
and 
\begin{align*}
	M^{\alpha}f(x)\leq V_{q}(\mathcal{A}^{\alpha}f)(x)+|A_{t_{0}}^{\alpha}f(x)|
\end{align*}
for any $t_{0}\in (0,\infty)$ and any $q\in (2,\infty)$. Hence, the $\Lp(\mathbb{R}^{d})$-boundedness of $V_{\infty}(\mathcal{A}^{\alpha})$ is equivalent to the $\Lp(\mathbb{R}^{d})$-boundedness of $M^{\alpha}$ and the $\Lp(\mathbb{R}^{d})$-boundedness of $V_{q}(\mathcal{A}^{\alpha})$ implies the $\Lp(\mathbb{R}^{d})$-boundedness of $M^{\alpha}$.

In \cite{Jones08Trans}, Jones, Seeger and Wright showed the following theorem, which improves the well-known estimates for the classical spherical maximal means.
\begin{thmx}\label{thm-Jones}
	The following inequality holds
	\begin{align}\label{eq-vari-sph-boun}
		\|V_{q}(\mathcal{A}^{0}f)\|_{\Lp(\mathbb{R}^{d})}\lesssim \|f\|_{\Lp(\mathbb{R}^{d})},
	\end{align}
if 
\begin{align*}
	q>2 \text{ and } \frac{d}{d-1}<p\leq 2d
\end{align*}
or
\begin{align*}
	q>\frac{p}{d} \text{ and } p>2d.
\end{align*}
For $p>2d$, \eqref{eq-vari-sph-boun} fails if $q<p/d$.
\end{thmx}
This conclusion is essentially sharp except for the critical case $q=p/d$ for $p>2d$. For $q=p/d$, $p>2d$ and $d\geq 3$, Beltran, Oberlin, Roncal, Seeger and Stovall \cite[Theorem 1.1]{Beltran2022MathAnn} proved that $V_{p/d}(\mathcal{A}^{0})$ maps $L^{p,1}(\mathbb{R}^{d})$ to $L^{p,\infty}(\mathbb{R}^{d})$. They conjectured that a similar endpoint result holds when $d=2$, and it remains open.

In this paper, we show the following theorem.
\begin{thm}\label{thm-main}
	Let $\alpha\in\mathbb{C}$, $q\in(2,\infty]$ and $p\in (1,\infty]$.
	\begin{enumerate}
	\item For $d\geq 2$, if $V_{q}(\mathcal{A}^{\alpha})$ is bounded on $\Lp(\mathbb{R}^{d})$, then one of the following conditions holds
		\begin{enumerate}
			\item [$(\text{a}_{1})$] $1<p<2$, $\Re \alpha >1-d+d/p$ and $q>2$;
			\item [$(\text{a}_{2})$] $2\leq p\leq 2d/(d-1)$, $\Re \alpha\geq1/p-(d-1)/2$ and $1/q\leq (d-1)/2+\Re \alpha$;
			\item [$(\text{a}_{3})$] $2d/(d-1)< p< \infty$, $\Re \alpha\geq (1-d)/p$ and $1/q\leq \Re\alpha+d/p$;
			\item [$(\text{a}_{4})$] $p=\infty$, $\Re \alpha\geq 0$ and $q=\infty$.
		\end{enumerate}
	\item For $d=2$, $V_{q}(\mathcal{A}^{\alpha})$ is bounded on $\Lp(\mathbb{R}^{2})$ if
	    \begin{enumerate}
	    	\item [$(\text{a}_{1})$] $1<p<2$, $\Re \alpha>-1+2/p$ and $q>2$;
	    	\item [$(\text{a}_{2}^{\prime})$] $2\leq p\leq 4$, $\Re \alpha>1/p-1/2$ and $1/q< 1/2+\Re \alpha$;
	    	\item [$(\text{a}_{3}^{\prime})$] $4< p<\infty$, $\Re \alpha> -1/p$ and $1/q< \Re\alpha+2/p$;
	    	\item [$(\text{a}_{4}^{\prime})$] $p=\infty$, $\Re \alpha> 0$ or $\alpha=0$, and $q=\infty$.
	    \end{enumerate}
    \item For $d>2$, $V_{q}(\mathcal{A}^{\alpha})$ is bounded on $\Lp(\mathbb{R}^{d})$ if
    	\begin{enumerate}
    		\item [$(\text{a}_{1})$] $1<p<2$, $\Re \alpha>1-d+d/p$ and $q>2$;
    		\item [$(\text{a}_{2}^{\prime \prime})$] $2\leq p\leq 2(d+1)/(d-1)$, $\Re \alpha>(1-d)/4+(3-d)/(2p)$ and $1/q<(d-1)/4+(d-1)/(2p)+\Re \alpha$;
    		\item [$(\text{a}_{3}^{\prime \prime})$] $2(d+1)/(d-1)< p<\infty$, $\Re \alpha> (1-d)/p$ and $1/q< \Re\alpha+d/p$;
    		\item [$(\text{a}_{4}^{\prime \prime})$] $p=\infty$, $\Re \alpha> 0$ or $\alpha=0$, and $q=\infty$.
    	\end{enumerate}
	\end{enumerate}
\end{thm}

Let $d\geq 2$. From Theorem \ref{thm-main}, it follows that $V_{q}(\mathcal{A}^{1})$ is bounded on $\Lp(\mathbb{R}^{d})$ when $q\in (2,\infty)$ and $p\in (1,\infty)$ by taking $\alpha=1$, which was established in \cite{Jones03Israel} for $p\in (1,2]$ and in \cite{Krause2018Ergodic,MR3671712} for $p\in (1,\infty)$. Taking $\alpha=0$, we reobtain Theorem \ref{thm-Jones}.

The relations of $\alpha$, $q$ and $p$ are summarized in following figures when $d=2$ .

\begin{figure}[H] 
	\centering 
	\begin{tikzpicture}
		\draw[thick,->]  (0,0) -- (6,0);
		\draw[thick,->] (0,0) -- (0,6);
		\draw[densely dashed] (1,0) -- (4,0) -- (4,4) -- (2.5,4) -- (1,1) -- (1,0);
     	\fill[color=lightgray] (1,0) -- (4,0) -- (4,4) -- (2.5,4) -- (1,1) -- (1,0);
		\node[left] at (0,6) {$\frac{1}{q}$};
		\node[left] at (0,1) {$-\Re \alpha$};
		\node[left] at (0,2.5) {$\frac{1}{4}$};
		\node[left] at (0,4) {$\frac{1+2\Re \alpha}{2}$};
		\node[left] at (0,5) {$\frac{1}{2}$}; 
		\node[right] at (6,0) {$\frac{1}{p}$};
		\node[below] at (1,0) {$-\Re \alpha$};
		\node[above right] at (1,0) {$Q_{1}$};
		\node[below] at (2.5,0) {$\frac{1}{4}$};
		\node[below] at (5,0) {$\frac{1}{2}$};
		\node[below] at (4,0) {$\frac{1+2\Re \alpha}{2}$};
		\node[above right] at (4,0) {$Q_{2}$};
		\node[above right] at (4,4) {$Q_{3}$};
		\node[left] at (4,2) {$\mathfrak{Q}_{1}$};
		\node[above right] at (2.5,4) {$Q_{4}$};
		\node[right] at (1,1) {$Q_{5}$};
		\draw (5,-2pt) -- (5,2pt);
		\draw (2.5,-2pt) -- (2.5,2pt);
		\draw (-2pt,2.5) -- (2pt,2.5);
		\draw (-2pt,5) -- (2pt,5);
		\draw (-2pt,1) -- (2pt,1);
		\draw (-2pt,4) -- (2pt,4);
		\node[circle,draw=black, fill=white, inner sep=0pt,minimum size=4pt] at (1,1) {};
		\node[circle,draw=black, fill=white, inner sep=0pt,minimum size=4pt] at (1,0) {};
		\node[circle,draw=black, fill=white, inner sep=0pt,minimum size=4pt] at (4,0) {};
		\node[circle,draw=black, fill=white, inner sep=0pt,minimum size=4pt] at (4,4) {};
		\node[circle,draw=black, fill=white, inner sep=0pt,minimum size=4pt] at (2.5,4) {};
     	\draw[densely dashed,red] (5,0) -- (5,5) -- (2.5,5) -- (0,0);
    	\node[left] at (0,0) {$P_{1}$};
    	\node[above right] at (5,0) {$P_{2}$};
    	\node[red,left] at (5,2.5) {$\mathfrak{Q}$};
    	\node[above right] at (5,5) {$P_{3}$};
    	\node[above right] at (2.5,5) {$P_{4}$};
    	\node[circle,draw=red, fill=red, inner sep=0pt,minimum size=4pt] at (0,0) {};
    	\node[circle,draw=red, fill=white, inner sep=0pt,minimum size=4pt] at (5,0) {};
    	\node[circle,draw=red, fill=white, inner sep=0pt,minimum size=4pt] at (2.5,5) {}; 
    	\node[circle,draw=red, fill=white, inner sep=0pt,minimum size=4pt] at (5,5) {};
	    \draw[densely dashed,blue] (2.5,0) -- (2.5,2.5);
	    \node[above right] at (2.5,0) {$O_{1}$};
	    \node[above right] at (2.5,2.5) {$O_{2}$};
	    \node[circle,draw=blue, fill=white, inner sep=0pt,minimum size=4pt] at (2.5,2.5) {};
	    \node[circle,draw=blue, fill=white, inner sep=0pt,minimum size=4pt] at (2.5,0) {};
	\end{tikzpicture}
	\caption{The pentagon $\mathfrak{Q}_{1}$ with vertices $Q_{1},Q_{2},Q_{3},Q_{4},Q_{5}$ for $\Re \alpha\in(-1/4,0)$. } 
	\label{Fig.main1} 
\end{figure}
$\textbf { Case (i)}$  When $\Re \alpha\in (-1/4,0)$, $V_{q}(\mathcal{A}^{\alpha})$ is bounded on $L^{p}(\mathbb{R}^{2})$ if $(1/p,1/q)$ lies in the interior of $\mathfrak{Q}_{1}$ or the open segment $Q_{1}Q_{2}$, and $V_{q}(\mathcal{A}^{\alpha})$ is unbounded if $(1/p,1/q)\notin \mathfrak{Q}_{1}$. The trapezoid $\mathfrak{Q}$ with vertices $P_{1},P_{2},P_{3},P_{4}$ corresponds to Theorem \ref{thm-Jones} (the case $\alpha=0$). The blue segment $O_{1}O_{2}$ corresponds to the case $\Re \alpha=-1/4$. The pentagon $\mathfrak{Q}_{1}$ tends to $\mathfrak{Q}$ when $\Re \alpha$ tends $0$, and tends to the segment $O_{1}O_{2}$ when $\Re \alpha$ tends to $-1/4$.

\begin{figure}[H] 
	\centering 
    \begin{tikzpicture}
    	\draw[thick,->]  (0,0) -- (11,0);
    	\draw[thick,->] (0,5) -- (0,6);
    	\draw[densely dashed] (0,0) -- (6.5,0) -- (6.5,5) -- (1,5) -- (0,3) -- (0,0);
        \fill[color=lightgray] (0,0) -- (6.5,0) -- (6.5,5) -- (1,5) -- (0,3) -- (0,0);
    	\node[left] at (0,6) {$\frac{1}{q}$};
    	\node[left] at (0,3) {$\Re \alpha$};
    	\node[left] at (0,5) {$\frac{1}{2}$};
    	\node[right] at (11,0) {$\frac{1}{p}$};
    	\node[below] at (1,0) {$\frac{1-2\Re \alpha}{4}$};
    	\node[below] at (5,0) {$\frac{1}{2}$};
     	\node[below] at (6.5,0) {$\frac{1+\Re \alpha}{2}$};
    	\node[below] at (10,0) {1};
    	\draw (10,-2pt) -- (10,2pt);
     	\draw (1,-2pt) -- (1,2pt);
     	\draw (5,-2pt) -- (5,2pt);
    	\draw (-2pt,5) -- (2pt,5);
    	\node[circle,draw=black, fill=black, inner sep=0pt,minimum size=4pt] at (0,0) {};
    	\node[above left] at (0,0) {$Q_{1}$};
    	\node[circle,draw=black, fill=white, inner sep=0pt,minimum size=4pt] at (6.5,0) {};
    	\node[above right] at (6.5,0) {$Q_{2}$};
    	\node[circle,draw=black, fill=white, inner sep=0pt,minimum size=4pt] at (6.5,5) {};
    	\node[left] at (6.5,2.5) {$\mathfrak{Q}_{2}$};
    	\node[above right] at (6.5,5) {$Q_{3}$};
    	\node[circle,draw=black, fill=white, inner sep=0pt,minimum size=4pt] at (1,5) {};
    	\node[above right] at (1,5) {$Q_{4}$};
    	\node[circle,draw=black, fill=white, inner sep=0pt,minimum size=4pt] at (0,3) {};
    	\node[right] at (0,3.2) {$Q_{5}$};
    \draw[densely dashed,red] (5,0) -- (5,5);
    \draw[densely dashed,red](2.5,5) -- (0,0);
    \node[above right] at (5,0) {$P_{1}$};
    \node[red,left] at (5,2.5) {$\mathfrak{Q}$};
    \node[above right] at (5,5) {$P_{2}$};
    \node[above right] at (2.5,5) {$P_{3}$};
    \node[circle,draw=red, fill=white, inner sep=0pt,minimum size=4pt] at (5,0) {};
    \node[circle,draw=red, fill=white, inner sep=0pt,minimum size=4pt] at (2.5,5) {}; 
    \node[circle,draw=red, fill=white, inner sep=0pt,minimum size=4pt] at (5,5) {};
    \draw[densely dashed,blue] (10,0) -- (10,5) -- (6.5,5);
    \draw[densely dashed,blue] (0,3) -- (0,5) -- (1,5);
    \node[above right] at (10,0) {$O_{1}$};
    \node[blue,left] at (10,2.5) {$\mathfrak{Q}_{3}$};
    \node[above right] at (10,5) {$O_{2}$};
    \node[above right] at (0,5) {$O_{3}$};
    \node[circle,draw=blue, fill=white, inner sep=0pt,minimum size=4pt] at (0,5) {}; 
    \node[circle,draw=blue, fill=white, inner sep=0pt,minimum size=4pt] at (10,0) {}; 
    \node[circle,draw=blue, fill=white, inner sep=0pt,minimum size=4pt] at (10,5) {}; 
    \end{tikzpicture}
	\caption{The pentagon $\mathfrak{Q}_{2}$ with vertices $Q_{1},Q_{2},Q_{3},Q_{4},Q_{5}$ for $\Re \alpha\in(0,1/2)$.} 
	\label{Fig.main2} 
\end{figure}
$\textbf { Case (ii)}$ When $\Re \alpha\in (0,1/2)$, $V_{q}(\mathcal{A}^{\alpha})$ is bounded on $L^{p}(\mathbb{R}^{2})$ if $(1/p,1/q)$ lies in the interior of $\mathfrak{Q}_{2}$ or the half open segment $Q_{1}Q_{2}$, and $V_{q}(\mathcal{A}^{\alpha})$ is unbounded if $(1/p,1/q)\notin \mathfrak{Q}_{2}$. The trapezoid $\mathfrak{Q}$ with vertices $Q_{1},P_{1},P_{2},P_{3}$ corresponds to Theorem \ref{thm-Jones} (the case $\alpha=0$). The rectangle $\mathfrak{Q}_{3}$ with vertices $Q_{1},O_{1},O_{2},O_{3}$ corresponds to the case $\Re \alpha=1$. The pentagon $\mathfrak{Q}_{2}$ tends to $\mathfrak{Q}$ when $\Re \alpha$ tends $0$, and tends to the rectangle $\mathfrak{Q}_{3}$ when $\Re \alpha$ tends to $1$. If $\Re \alpha\geq 1$, then $V_{q}(\mathcal{A}^{\alpha})$ is bounded on $L^{p}(\mathbb{R}^{2})$ if $(1/p,1/q)$ lies in the interior of the rectangle $\mathfrak{Q}_{3}$ or the half open segment $Q_{1}O_{1}$.

As a corollary of Theorem \ref{thm-main}, we have similar conclusions for $\delta$-jump operator, which is defined as
\begin{align*}
	\Lambda_{\delta}(\mathcal{A}^{\alpha}f)(x)
	&:=\sup\{N\in \mathbb{N}: \text{there exist } s_{1}<t_{1}\leq s_{2}<t_{2}<\cdots\leq s_{N}<t_{N} \\
	&\ \ \ \ \ \ \ \ \ \ \ \ \ \ \ \ \ \ \ \ \ \ \text{such that } |A^{\alpha}_{t_{l}}f(x)-A^{\alpha}_{s_{l}}f(x)|>\delta, l=1,2,...,N\}.	
\end{align*}
Observe that for any $\delta>0$, $q\in (2,\infty)$ and $x\in \mathbb{R}^{d}$,
\begin{align*}
	\delta[\Lambda_{\delta}(\mathcal{A}^{\alpha}f)(x)]^{1/q}\leq V_{q}(\mathcal{A}^{\alpha}f)(x),
\end{align*}
from which and a similar argument in Section \ref{sec-nec} we obtain the following result.
\begin{cor}
	Given $\delta>0$, let $q\in (2,\infty)$ and $p>1$.
	\begin{enumerate}
		\item For $d\geq 2$, if $\delta[\Lambda_{\delta}(\mathcal{A}^{\alpha}f)(x)]^{1/q}$ is bounded on $\Lp(\mathbb{R}^{d})$, then one of the following conditions holds
		\begin{enumerate}
			\item [$(\text{a}_{1})$] $1<p<2$, $\Re \alpha >1-d+d/p$ and $q\in (2,\infty)$;
			\item [$(\text{a}_{2})$] $2\leq p\leq 2d/(d-1)$, $\Re \alpha\geq1/p-(d-1)/2$ and $1/q\leq (d-1)/2+\Re \alpha$;
			\item [$(\text{a}_{3})$] $2d/(d-1)< p< \infty$, $\Re \alpha\geq (1-d)/p$ and $1/q\leq \Re\alpha+d/p$.
		\end{enumerate}
		\item For $d=2$, $\delta[\Lambda_{\delta}(\mathcal{A}^{\alpha}f)(x)]^{1/q}$ is bounded on $\Lp(\mathbb{R}^{2})$ if
		\begin{enumerate}
			\item [$(\text{a}_{1})$] $1<p<2$, $\Re \alpha>-1+2/p$ and $q>2$;
			\item [$(\text{a}_{2}^{\prime})$] $2\leq p\leq 4$, $\Re \alpha>1/p-1/2$ and $1/q< 1/2+\Re \alpha$;
			\item [$(\text{a}_{3}^{\prime})$] $4< p<\infty$, $\Re \alpha> -1/p$ and $1/q< \Re\alpha+2/p$.
		\end{enumerate}
		\item For $d>2$, $\delta[\Lambda_{\delta}(\mathcal{A}^{\alpha}f)(x)]^{1/q}$ is bounded on $\Lp(\mathbb{R}^{d})$ if
		\begin{enumerate}
			\item [$(\text{a}_{1})$] $1<p<2$, $\Re \alpha>1-d+d/p$ and $q>2$;
			\item [$(\text{a}_{2}^{\prime \prime})$] $2\leq p\leq 2(d+1)/(d-1)$, $\Re \alpha>(1-d)/4+(3-d)/(2p)$ and $1/q<(d-1)/4+(d-1)/(2p)+\Re \alpha$;
			\item [$(\text{a}_{3}^{\prime \prime})$] $2(d+1)/(d-1)< p<\infty$, $\Re \alpha> (1-d)/p$ and $1/q< \Re\alpha+d/p$.
		\end{enumerate}
	\end{enumerate}
\end{cor}

To prove Theorem \ref{thm-main} (i) when $\alpha=0$, it is sufficient to construct only one example to show Theorem \ref{thm-main} (i) (a$_{3}$) since $1/q\leq (d-1)/2$ always holds, see \cite{Jones08Trans}. But for general $\alpha\in \mathbb{C}$, we have to construct examples to verify (a$_{2}$)-(a$_{4}$). For (a$_{4}$), we intend to show $V_{q}(\mathcal{A}^{\alpha})$ is unbounded on $L^{\infty}(\mathbb{R}^{d})$ when $\Re \alpha>0$ and $q\in (2,\infty)$. We mention that this result is new even for $\alpha=1$. In \cite{Krause2018Ergodic,MR3671712}, the authors independently showed $V_{q}(\mathcal{A}^{1})$ is bounded from $L_{c}^{\infty}(\mathbb{R}^{d})$ to $BMO(\mathbb{R}^{d})$ when $q\in (2,\infty)$, where $L_{c}^{\infty}(\mathbb{R}^{d})$ denotes the space of bounded measurable functions with compact support. Here, we consider the behavior of $V_{q}(\mathcal{A}^{1})$ on $L^{\infty}(\mathbb{R}^{d})$.

To show Theorem \ref{thm-main} (i) (a$_{3}$), we employ the main idea from \cite{Jones08Trans}. Note that
\begin{align}\label{eq-four-dsigma}
	\widehat{d\sigma}(\lambda x)=\int_{\mathbb{S}^{d-1}}e^{2\pi i \lambda x\cdot \theta}d\sigma(\theta)
\end{align}
is essentially a constant when $\lambda |x|$ is small, we obtain the main term of $A_{t}^{\alpha}f_{\lambda}^{\alpha}(x)$ (see Section \ref{sec-nec} for its definition) when $\lambda|x|$ is small and $t\approx 1$, from which Theorem \ref{thm-main} (i) (a$_{3}$) follows. To prove Theorem \ref{thm-main} (i) (a$_{2}$), we consider $|x|\approx 3$, then we use the asymptotic expansion for Bessel function to decompose \eqref{eq-four-dsigma} into the main term and the error term. Then, we choose a special sequence $\{t_{n}\}_{n}$ and obtain the main term of $A_{t_{n}}^{\alpha}f_{\lambda}^{\alpha}(x)$, which implies the desired estimates. About Theorem \ref{thm-main} (i) (a$_{4}$), we first consider the simpler case $\alpha=1$. For $n\in \mathbb{N}$, we select $f_{n}\in L^{\infty}(\mathbb{R}^{d})$ and $\{t_{j}\}_{j=1}^{n}$ such that $A_{t_{j}}^{1}f_{n}(0)\approx1$ when $j$ is even and $A_{t_{j}}^{1}f_{n}(0)=0$ when $j$ is odd. Next, we show $|A_{t_{j}}^{1}f_{n}(x)-A_{t_{j}}^{1}f_{n}(0)|$ is small when $|x|$ is small, hence $V_{q}(\mathcal{A}^{1})$ is unbounded on $L^{\infty}(\mathbb{R}^{d})$. If $\Re \alpha>0$ and $\alpha\neq 1$, we can not repeat the above argument since it is difficult to calculate $A_{t}^{\alpha}f(0)$ directly. So, we choose $f_{n}\in L^{\infty}(\mathbb{R}^{d})$ and $\{t_{j}\}_{j=1}^{n}$ such that $|A_{t_{j}}^{\alpha}f_{n}(0)-A_{t_{j}}^{1}f_{n}(0)|$ is small, which enables us to deduce that $V_{q}(\mathcal{A}^{\alpha})$ is unbounded on $L^{\infty}(\mathbb{R}^{d})$ when $\Re \alpha>0$. 

The proofs of Theorem \ref{thm-main} (ii) and (iii) are based on the arguments in \cite{Jones08Trans}. We first decompose the variation operator into the long variation operator and the short variation operator. It is unclear if Theorem 1.1 in \cite{Jones08Trans}  can be applied directly to obtain the $L^{p}(\mathbb{R}^{d})$-boundedness of the long variation operator since it is difficult to verify that $A_{t}^{\alpha}f=f*\sigma_{t}^{\alpha}$ holds for some finite Borel measure $\sigma^{\alpha}$ with compact support when $\Re \alpha\leq 0$ with $\alpha\neq 0$. Hence, we use the discrete square function estimates established in \cite{beltran2022multiscale} to study the long variation operator. Next, we employ the local smoothing estimates and the square function estimates to get the $L^{p}(\mathbb{R}^{d})$-boundednss of the short variation operator.

The article is organized as follows. We first recall some well known results including a multiplier theorem and some variational inequalities in Section \ref{sec-pre}. In Section \ref{sec-bound}, we prove Theorem \ref{thm-main} (ii) and (iii). Finally, we provide examples to show Theorem \ref{thm-main} (i) in Section \ref{sec-nec}.

Throughout this article, each different appearance of the letter $C$ may represent a different positive constant and is independent of the main parameters. We write $A\lesssim B$ if there is $C>0$ such that $A\leq CB$, and write $A\approx B$ when $A\lesssim B \lesssim A$. We use $\widehat{f} $ and $\mathcal{F}^{-1}(f)$ to denote the Fourier transform of $f$ and the reverse Fourier transform of $f$, respectively. For any $E\subset \mathbb{R}^{d}$, we denote by $\chi_{E}$ the characteristic function of $E$.

\section{Preliminaries}\label{sec-pre}
In this section, we introduce some definitions and facts which will be used in the proof of Theorem \ref{thm-main}. We first recall the following asymptotic expansion for the Bessel function $J_\beta$ when $\beta\in \mathbb{C}$ (see \cite[p.~199]{bookWatson}), 
\begin{align}\label{al-Bessel}
	J_\beta(r)=r^{-1 / 2} e^{i r}  [b_{0,\beta}+R_{1,\beta}(r)]+r^{-1 / 2} e^{-i r} [d_{0,\beta}+R_{2,\beta}(r)], \quad r \geq 1,
\end{align}
where $b_{0,\beta}$ and $d_{0,\beta}$ are some suitable coefficients, and $R_{1,\beta}(r)$ and $ R_{2,\beta}(r)$ satisfy the following inequality for any $N\in \mathbb{N}$,
\begin{align*}
	\left|\left(\frac{d}{dr}\right)^{N}R_{1,\beta}(r)\right|+\left|\left(\frac{d}{dr}\right)^{N}R_{2,\beta}(r)\right|\lesssim r^{-N-1}, \quad r\geq 1.
\end{align*}

Next, we introduce the definitions of the long variation operator and the short variation operator. For $q\in [2,\infty)$ and each $j\in \mathbb{Z}$, let
\begin{align*}
	V_{q, j}(\mathcal{A}^{\alpha} f)(x):=\sup_{N\in\mathbb{N}}\sup _{\substack{t_1<\cdots<t_N \\\{t_{l}\}_{l=1}^{N} \subset[2^j, 2^{j+1}]}}\left( \sum_{l=1}^{N-1}\left|A_{t_{l+1}}^{\alpha} f(x)-A_{t_{l}}^{\alpha} f(x)\right|^{q}\right)^{1 / q}.
\end{align*}
We define the short variation operator 
\begin{align*}
	V_{q}^{\text{\rm sh}}(\mathcal{A}^{\alpha} f)(x):=\left(\sum_{j \in \mathbb{Z}}\left[V_{q, j}(\mathcal{A}^{\alpha} f)(x)\right]^{q}\right)^{1 / q}
\end{align*}
and the long variation operator
\begin{align*}
	V_{q}^{\text{\rm dyad}}(\mathcal{A}^{\alpha}f)(x):=\sup_{N\in \mathbb{N}}\sup _{\substack{t_1<\cdots<t_N \\ \{t_{l}\}_{l=1}^{N} \subset \mathbb{Z}}}\left( \sum_{l=1}^{N-1}\Big|A_{2^{t_{l+1}}}^{\alpha} f(x)-A_{2^{t_{l}}}^{\alpha} f(x)\Big|^{q}\right)^{1 / q}.
\end{align*}
Based on the following lemma (see \cite[Lemma 1.3]{Jones08Trans}), the estimate of $V_{q}(\mathcal{A}^{\alpha})$ is reduced to those of the long variation operator and the short variation operator.
\begin{lem}\label{lem-varia-decom}
	For $q\in (2,\infty)$,
	\begin{align*}
		V_{q}(\mathcal{A}^{\alpha}f)(x)\lesssim V_{q}^{\text{\rm sh}}(\mathcal{A}^{\alpha}f)(x)+V_{q}^{\text{\rm dyad}}(\mathcal{A}^{\alpha}f)(x).
	\end{align*}
\end{lem}

We now recall the following multiplier lemma, which was stated in \cite[p.~6737]{Jones08Trans}; see also \cite[Section 7]{Guo2020Anal}.
\begin{lem}\label{lem-local}
	Let $\left\{m_{s}: s \in [1,2]\right\}$ be a family of Fourier multipliers on $\mathbb{R}^{d}$, each of which is compactly supported on $\left\{\xi: 1/8 \leq|\xi| \leq 2\right\}$ and satisfies
	$$
	\sup _{s \in [1,2]}\left|\partial_{\xi}^{\tau} m_{s}(\xi)\right| \leq B \quad \text { for each } 0 \leq|\tau| \leq d+1
	$$
	for some positive constant $B$. Assume that there exists some positive constant $A$ such that
	\begin{align*}
		\sup_{j\in \mathbb{Z}}\left\|\left(\int_{1}^{2}\left|\mathcal{F}^{-1}[m_{s}(2^{j}\cdot)\widehat{f}(\cdot)]\right|^{2}ds\right)^{1/2}\right\|_{L^{2}(\mathbb{R}^{d})}\leq A\|f\|_{L^{2}(\mathbb{R}^{d})}
	\end{align*}
and
    \begin{align*}
    	\sup_{j\in \mathbb{Z}}\left\|\left(\int_{1}^{2}\left|\mathcal{F}^{-1}[m_{s}(2^{j}\cdot)\widehat{f}(\cdot)]\right|^{2}ds\right)^{1/2}\right\|_{\Lp(\mathbb{R}^{d})}\leq A\|f\|_{\Lp(\mathbb{R}^{d})}
    \end{align*}
for some $p \in(2, \infty)$. 
Then
	\begin{align*}
		\left\|\left(\sum_{j\in \mathbb{Z}}\int_{1}^{2}\left|\mathcal{F}^{-1}[m_{s}(2^{j}\cdot)\widehat{f}(\cdot)]\right|^{2}ds\right)^{1/2}\right\|_{L^{p}(\mathbb{R}^{d})}\lesssim A\left|\log \left(2+B/A\right)\right|^{1 / 2-1 / p}\|f\|_{L^p(\mathbb{R}^d)}.
	\end{align*}
\end{lem}

We will also use the following estimates of square function associated to multipliers; see \cite[Lemma 4]{Rubio1986Duke}.
\begin{lem}\label{lem-g-function}
	Suppose that $\mu\in C^{s}(\mathbb{R}^{d})$ for some integer $s>d/2$ and supported in the anulus $\{\xi:1/2\leq |\xi| \leq 2\}$. Let 
	\begin{align*}
		G_{\mu}f(x):=\left(\int_{0}^{\infty}\left|\mathcal{F}^{-1}[\widehat{f}(\cdot)\mu(t\cdot)](x)\right|^{2}\frac{dt}{t}\right)^{1/2}.
	\end{align*} 
If $\gamma>d/2$, then
\begin{align*}
	\|G_{\mu}f\|_{L^{1}(\mathbb{R}^{d})}\lesssim \|\mu\|_{L_{\gamma}^{2}(\mathbb{R}^{d})}\|f\|_{H^{1}(\mathbb{R}^{d})}
\end{align*}
and
\begin{align*}
	\|G_{\mu}f\|_{BMO(\mathbb{R}^{d})}\lesssim \|\mu\|_{L_{\gamma}^{2}(\mathbb{R}^{d})}\|f\|_{L^{\infty}(\mathbb{R}^{d})},
\end{align*}
where $\|\cdot\|_{L_{\gamma}^{2}(\mathbb{R}^{d})}$ denotes the Sobolev norm.
\end{lem}

The following lemma was established in \cite{Jones08Trans}.
\begin{lem}\label{lem-Jones-mea}
	Suppose that $\sigma$ is a finite Borel measure on $\mathbb{R}^{d}$ with compact support and satisfies
	\begin{align*}
		|\hat{\sigma}(\xi)|\lesssim |\xi|^{-b}, \text{ for some $b>0$}.
	\end{align*}
Then for $q\in (2,\infty)$ and $p\in (1,\infty)$,
	\begin{align*}
		\|V_{q}(\{f*\sigma_{{k}}:k\in \mathbb{Z}\})\|_{\Lp(\mathbb{R}^{d})}\lesssim\|f\|_{\Lp(\mathbb{R}^{d})},
	\end{align*}
	where $\sigma_{{k}}$ is defined by
	\begin{align*}
		\left\langle\sigma_{{k}},f\right\rangle :=\int_{\mathbb{R}^{d}}f(2^{k}x)d\sigma(x).
	\end{align*}
\end{lem}
By this lemma, we immediately get the following corollary.
\begin{cor}\label{cor-var-long}
	Let $g$ be a smoothing function and $\supp g\subset [-1,1]^{d}$. Then for $q\in (2,\infty)$ and $p\in (1,\infty)$,
	\begin{align*}
		\|V_{q}(\{f*g_{k}:k\in \mathbb{Z}\})\|_{\Lp(\mathbb{R}^{d})}\lesssim\|f\|_{\Lp(\mathbb{R}^{d})},
	\end{align*}
	where
	\begin{align*}
		g_{k}(x):=2^{-kd}g(2^{-k}x).
	\end{align*}
\end{cor}

Finally, we state a Sobolev embedding theorem for $q$-variations; see \cite[Proposition 2.2]{Guo2020Anal}, \cite[p.~6729]{Jones08Trans} and \cite[Lemma B.1]{Krause2018Ergodic} for further details.
\begin{lem}\label{lem-variation-norm}
	Suppose $F\in C^{1}(\mathbb{R}^{+})$. Then for $q\in [1,\infty)$,
	\begin{align*}
		\|\{F(t)\}_{t>0}\|_{v_{q}}\lesssim \|F\|_{L^{q}(\mathbb{R}^{+})}^{1 / {q}^{\prime}}\left\|F^{\prime}\right\|_{L^{q}(\mathbb{R}^{+})}^{1 / q}.
	\end{align*}
\end{lem}

\section{The sufficient part in Theorem \ref{thm-main}}\label{sec-bound}
This section is devoted to the proofs of Theorem \ref{thm-main} (ii) and (iii). By Lemma \ref{lem-varia-decom}, it is sufficient to estimate  $\|V_{q}^{\text{\rm sh}}(\mathcal{A}^{\alpha}f)\|_{L^{p}(\mathbb{R}^{d})}$ and $\|V_{q}^{\text{\rm dyad}}(\mathcal{A}^{\alpha}f)\|_{L^{p}(\mathbb{R}^{d})}$. 

Choose $\psi_{0}\in \mathcal{S}(\mr^{d})$ such that 
\begin{align*}
	\psi_{0}(\xi)=\begin{cases}
		1, &\text{if } |\xi|\leq 1;\\
		0,&\text{if } |\xi|\geq 2.
	\end{cases}
\end{align*}
For $k\geq 1$, define
\begin{align*}
	\psi_{k}(\xi):=\psi_{0}(2^{-k}\xi)-\psi_{0}(2^{-(k-1)}\xi).
\end{align*}
Obviously, $\supp \psi_{k} \subset \{\xi:2^{k-1}\leq |\xi|\leq 2^{k+1}\}$ for $k\geq 1$ and 
\begin{align*}
	\sum_{k=0}^{\infty}\psi_{k}(\xi)=1
\end{align*}
for any $\xi \in \mr^{d}$.
Thus, 
\begin{align*}
	\widehat{A_{t}^{\alpha}f}(\xi)=\widehat{f}(\xi)m^{\alpha}(t\xi)=\sum_{k=0}^{\infty}\widehat{f}(\xi)m^{\alpha}(t\xi)\psi_{k}(t\xi),
\end{align*}
where $m^{\alpha}$ is as in \eqref{eq-average-def}. Let 
\begin{align*}
	m_{k}^{\alpha}(\xi):=m^{\alpha}(\xi)\psi_{k}(\xi) \text{ and }\widehat{A_{t,k}^{\alpha}f}(\xi):=\widehat{f}(\xi)m_{k}^{\alpha}(t\xi).
\end{align*}
In the following, we estimate $\|V_{q}^{\text{\rm sh}}(\mathcal{A}^{\alpha}_{k}f)\|_{L^{p}(\mathbb{R}^{d})}$ and $\|V_{q}^{\text{\rm dyad}}(\mathcal{A}^{\alpha}_{k}f)\|_{L^{p}(\mathbb{R}^{d})}$ for $k\geq 0$, where $$\mathcal{A}_{k}^{\alpha}f(x):=\{A_{t,k}^{\alpha}f(x)\}_{t>0}.$$ 

We first consider the long variation operator.
\begin{lem}\label{lem-long-comp}
	Let $\alpha\in \mathbb{C}$, $q\in (2,\infty)$ and $p\in (1,\infty)$. Then
	\begin{align*}
		\|V_{q}^{\text{\rm dyad}}(\mathcal{A}^{\alpha}_{0}f)\|_{\Lp(\mathbb{R}^{d})}\lesssim \|f\|_{\Lp(\mathbb{R}^{d})}.
	\end{align*}
\end{lem}
\begin{proof}
	Choose $\phi^{\alpha}\in C_{c}^{\infty}(\mathbb{R}^{d})$ satisfying $\supp \phi\subset [-1,1]^{d}$ and $\widehat{\phi^{\alpha}}(0)=m^{\alpha}(0)$. Let $\widehat{\sigma^{\alpha}_{k}}(\xi):=m^{\alpha}(2^{k}\xi)\psi_{0}(2^{k}\xi)$ and $\widehat{\phi^{\alpha}_{k}}(\xi):=\widehat{\phi^{\alpha}}(2^{k}\xi)$.
	By the sublinearity of the variation operator, for any $x\in \mathbb{R}^{d}$, we have
	\begin{align}\label{eq-long-0}
		V_{q}^{\text{\rm dyad}}(\mathcal{A}^{\alpha}_{0}f)(x)\lesssim V_{q}(\{f*\phi_{k}^{\alpha}:k\in \mathbb{Z}\})(x)+V_{q}(\{f*(\phi_{k}^{\alpha}-\sigma^{\alpha}_{k}):k\in\mathbb{Z}\})(x).
	\end{align}
    Using Corollary \ref{cor-var-long}, the following inequality holds for $q\in (2,\infty)$ and $p\in (1,\infty)$,
    \begin{align}\label{eq-long-comp}
    	\|V_{q}(\{f*\phi_{k}^{\alpha}:k\in \mathbb{Z}\})\|_{\Lp(\mathbb{R}^{d})}\lesssim \|f\|_{\Lp(\mathbb{R}^{d})}.
    \end{align}

Since 
\begin{align*}
	\left|\widehat{\sigma^{\alpha}_{k}}(\xi)- \widehat{\phi^{\alpha}_{k}}(\xi)\right|\lesssim \min\{|2^{k}\xi|,|2^{k}\xi|^{-1}\}
\end{align*}
and
\begin{align*}
	\left|\sigma^{\alpha}_{k}(x)-\phi^{\alpha}_{k}(x)\right|\leq  \frac{C_{N}2^{-kd}}{(1+|2^{-k}x|)^{N}}
\end{align*}
for any $N\in \mathbb{N}$. From Theorem B in \cite{Javier1986Invent}, we deduce that
\begin{align}\label{eq-long-square}
	\left\|\left(\sum_{k=-\infty}^{\infty}\left|f*(\phi_{k}^{\alpha}-\sigma^{\alpha}_{k})\right|^{2}\right)^{1/2}\right\|_{L^{p}(\mathbb{R}^{d})}\lesssim \|f\|_{L^{p}(\mathbb{R}^{d})}
\end{align}
for any $p\in (1,\infty)$. By \eqref{eq-long-square} and the fact $q>2$, we obtain
\begin{align}
	\|V_{q}(\{f*(\phi_{k}^{\alpha}-\sigma^{\alpha}_{k}):k\in\mathbb{Z}\})\|_{\Lp(\mathbb{R}^{d})}\lesssim \|f\|_{\Lp(\mathbb{R}^{d})}. \label{eq-long-vanish}
\end{align}
Combining \eqref{eq-long-0}, \eqref{eq-long-comp} and \eqref{eq-long-vanish}, we complete the proof.
\end{proof}

\begin{prop}\label{prop-long-varia}
Let $\Re \alpha\in ((1-d)/2,0)$, $q\in (2,\infty)$ and $1/p\in (\Re \alpha/(1-d), 1+\Re \alpha/(d-1))$. Then
\begin{align*}
		\|V_{q}^{\text{\rm dyad}}(\mathcal{A}^{\alpha}f)\|_{L^{p}(\mathbb{R}^{d})}\lesssim \|f\|_{L^{p}(\mathbb{R}^{d})}.
\end{align*}
If $\Re \alpha\in [0,\infty)$, the above inequality holds for $q\in (2,\infty)$ and $p\in (1,\infty)$.
\end{prop}
\begin{proof}
We first consider $\Re \alpha\in ((1-d)/2,0)$. Recall that for $k>0$ and any $\epsilon>0$, the following inequality holds for all $p\in (1,\infty)$, 
	\begin{align}\label{eq-long-big}
		\left\|\left(\sum_{l\in \mathbb{Z}}\left|A_{2^{l},k}^{\alpha}f(x)\right|^{2}\right)^{1/2}\right\|_{L^{p}(\mathbb{R}^{d})}\lesssim 2^{k\left[(d-1)|1/p-1/2|+((1-d)/2-\Re \alpha)+\epsilon\right]}\|f\|_{L^{p}(\mathbb{R}^{d})};
	\end{align} 
see \cite[p.~71]{beltran2022multiscale}. Thus, by \eqref{eq-long-big}, Lemma \ref{lem-long-comp} and the fact $q>2$, we have
	\begin{align*}
		\|V_{q}^{\text{\rm dyad}}(\mathcal{A}^{\alpha}f)\|_{L^{p}(\mathbb{R}^{d})}&\leq \sum_{k=0}^{\infty}\|V_{q}^{\text{\rm dyad}}(\mathcal{A}^{\alpha}_{k}f)\|_{L^{p}(\mathbb{R}^{d})}\lesssim\|f\|_{L^{p}(\mathbb{R}^{d})}.
	\end{align*}
Similarly, this inequality holds for $q\in (2,\infty)$ and $p\in (1,\infty)$ when $\Re \alpha \in [0,\infty)$.
\end{proof}

Now we consider the short variation operator. The following lemma is a particular case of \cite[Lemma 6.1]{Jones08Trans}, we sketch the proof here for the reader's convenience.
\begin{lem}\label{lem-variation-S2}
	Let $\Re \alpha\in (1-d/2,1)$, $1/p \in ((1-\Re \alpha)/d,(d-1+\Re \alpha)/d)$. Then 
	\begin{align*}
		\|V_{2}^{\text{\rm sh}}(\mathcal{A}^{\alpha}f)\|_{\Lp(\mr^{d})}\lesssim \|f\|_{\Lp(\mathbb{R}^{d})},
	\end{align*}
and this inequality holds for all $p\in (1,\infty)$ whenever $\Re \alpha\geq 1$.
\end{lem}
\begin{proof}
	Note that 
	\begin{align*}
		\|V_{2}^{\text{\rm sh}}(\mathcal{A}^{\alpha}f)\|_{\Lp(\mr^{d})}\lesssim \sum_{k=0}^{\infty}\|V_{2}^{\text{\rm sh}}(\mathcal{A}_{k}^{\alpha}f)\|_{\Lp(\mr^{d})}.
	\end{align*}
Hence, it suffices to estimate $V_{2}^{\text{\rm sh}}(\mathcal{A}_{k}^{\alpha}f)$ for $k\geq 0$. By Lemma \ref{lem-variation-norm} and the fact $q=2$, we have
	\begin{align*}
		V_{2}^{\text{\rm sh}}(\mathcal{A}_{k}^{\alpha}f)(x)&\lesssim 2^{k/2} \left(\int_{0}^{\infty}\left|A_{t,k}^{\alpha}f(x)\right|^{2}\frac{dt}{t}\right)^{1/2}+2^{-k/2} \left(\int_{0}^{\infty}\left|t \frac{\partial}{\partial t}A_{t,k}^{\alpha}f(x)\right|^{2}\frac{dt}{t}\right)^{1/2}\\
	    &=:2^{k/2}G_{k}^{\alpha}f(x)+2^{-k/2}\widetilde{G}_{k}^{\alpha}f(x).
	\end{align*}
By the definition of $A_{t,k}^{\alpha}f$, it is easy to see that
    \begin{align*}
    	t\frac{\partial}{\partial t}\widehat{A_{t,k}^{\alpha}f}(\xi)=\widehat{f}(\xi)\widetilde{m}_{k}^{\alpha}(t\xi),
    \end{align*}
    where $\widetilde{m}_{k}^{\alpha}(\xi):=\xi \cdot \nabla m_{k}^{\alpha}(\xi)$. From \eqref{al-Bessel} and the fact that $\supp \psi_{0}\subset \{\xi: |\xi|\leq 2\}$ and $\supp \psi_{k}\subset \{\xi:2^{k-1}\leq |\xi|\leq 2^{k+1}\}$ for $k\geq 1$, by Plancherel's theorem, we obtain
	\begin{align*}
		\|G_{k}^{\alpha}f\|_{L^{2}(\mr^{d})}\lesssim 2^{-k[(d-1)/2+\Re \alpha]}\|f\|_{L^{2}(\mr^{d})}
	\end{align*}
    and
    \begin{align*}
    	\left\|\widetilde{G}_{k}^{\alpha}f\right\|_{L^{2}(\mr^{d})}\lesssim 2^{-k[(d-3)/2+\Re \alpha]}\|f\|_{L^{2}(\mr^{d})},
    \end{align*}
    which implies
	\begin{align*}
		\|V_{2}^{\text{\rm sh}}(\mathcal{A}_{k}^{\alpha}f)\|_{L^{2}(\mr^{d})}\lesssim 2^{-k[(d-2)/2+\Re \alpha]}\|f\|_{L^{2}(\mr^{d})}.
	\end{align*}

Let $\mu_{k}^{\alpha}(\xi):=m_{k}^{\alpha}(2^{k}\xi)$ and $\widetilde{\mu}_{k}^{\alpha}(\xi):=\widetilde{m}_{k}^{\alpha}(2^{k}\xi)$. Applying Lemma \ref{lem-g-function} with $\mu_{k}^{\alpha}$ and $\widetilde{\mu}_{k}^{\alpha}$, we get
    \begin{align*}
    	\|G_{k}f\|_{L^{1}(\mr^{d})}\lesssim 2^{k(1/2-\Re \alpha+\epsilon)}\|f\|_{H^{1}(\mr^{d})}  
    \end{align*}
     and
     \begin{align*}
     	\left\|\widetilde{G}_{k}f\right\|_{L^{1}(\mr^{d})}\lesssim 2^{k(3/2-\Re \alpha+\epsilon)}\|f\|_{H^{1}(\mr^{d})}
     \end{align*}
for every $\epsilon>0$. By a similar argument, we have
     \begin{align*}
     	\|G_{k}f\|_{BMO(\mr^{d})}\lesssim 2^{k(1/2-\Re \alpha+\epsilon)}\|f\|_{L^{\infty}(\mr^{d})} 
     \end{align*}
     and
     \begin{align*}
     	\left\|\widetilde{G}_{k}f\right\|_{BMO(\mr^{d})}\lesssim 2^{k(3/2-\Re \alpha+\epsilon)}\|f\|_{L^{\infty}(\mr^{d})}
     \end{align*}
for every $\epsilon>0$. Interpolating with above estimates, we complete the proof. 
\end{proof}
From the above proof, we have the following corollary.
\begin{cor}\label{cor-short-var}
	For $\alpha\in \mathbb{C}$ and $p\in (1,\infty)$, 
	\begin{align*}
		\|V_{2}^{\text{\rm sh}}(\mathcal{A}^{\alpha}_{0}f)\|_{\Lp(\mr^{d})}\lesssim \|f\|_{\Lp(\mathbb{R}^{d})}.
	\end{align*}
\end{cor}
Combining Proposition \ref{prop-long-varia}, Lemmas \ref{lem-varia-decom} and \ref{lem-variation-S2}, we obtain some results in Theorem \ref{thm-main} (ii) and (iii). However, this argument does not cover the full range of $p$ and $q$ in Theorem \ref{thm-main} (ii) and (iii). If $\Re \alpha\leq 1-d/2$ or $0<\Re \alpha<1$ and $1/p\leq (1-\Re \alpha)/d$, we need to establish the following lemmas to prove Theorem \ref{thm-main} (ii) and (iii). 
\begin{lem}\label{lem-smoothing}
	Let $k\geq 1$ and
	\begin{align*}
		\mathrm{I}:=\left\|\left(\int_{0}^{\infty}\left|A_{t,k}^{\alpha}f(\cdot)\right|^{p}\frac{dt}{t}\right)^{1/p}\right\|_{\Lp(\mr^{d})}+2^{-k}\left\|\left(\int_{0}^{\infty}\left|t\frac{\partial }{\partial t }A_{t,k}^{\alpha}f(\cdot)\right|^{p}\frac{dt}{t}\right)^{1/p}\right\|_{\Lp(\mr^{d})}.
	\end{align*} 
    For any $\epsilon>0$, 
	\begin{enumerate}
		\item if $d=2$ and $p\in [4,\infty)$, or $d\geq 3$ and $p\in [2(d+1)/(d-1),\infty)$, then 
		\begin{align}
			\mathrm{I} \lesssim 2^{k(-d/p-\Re \alpha+\epsilon)}\|f\|_{\Lp(\mr^{d})};\label{eq-smooth-p4}
		\end{align}
		\item if $p\in [2,4]$, then
		\begin{align*}
			\mathrm{I} \lesssim 2^{k(-1/2-\Re \alpha+\epsilon)}\|f\|_{\Lp(\mr^{2})}; 
		\end{align*}
	    \item if $d\geq 3$ and $p\in [2,2(d+1)/(d-1)]$, then
	    \begin{align*}
	     \mathrm{I} \lesssim 2^{k[(1-d)(1/4+1/(2p))-\Re \alpha+\epsilon]}\|f\|_{\Lp(\mr^{d})}. 
	    \end{align*}
	\end{enumerate} 
\end{lem}
\begin{proof} 
	We first show (i). For $j\in \mathbb{Z}$, define
	\begin{align*}
		\beta_{j}(x):=\psi_{0}(2^{-j}x)-\psi_{0}(2^{-j+1}x).
	\end{align*}
	Thus, for any $x\neq 0$, we have
	\begin{align*}
		\sum_{k=-\infty}^{\infty}\beta_{j}(x)=1.
	\end{align*}
	The Littlewood-Paley operator $P_{j}$ is given by
	\begin{align*}
		P_{j}f:=\mathcal{F}^{-1}(\beta_{j}\widehat{f}).
	\end{align*}
	We claim that \eqref{eq-smooth-p4} follows from the estimate 
	\begin{align}\label{eq-local-Lp}
		&\left\|\left(\int_{1}^{2}\left|A_{t,k}^{\alpha}f(\cdot)\right|^{p}\frac{dt}{t}\right)^{1/p}\right\|_{\Lp(\mr^{d})}+2^{-k}\left\|\left(\int_{1}^{2}\left|t\frac{\partial }{\partial t }A_{t,k}^{\alpha}f(\cdot)\right|^{p}\frac{dt}{t}\right)^{1/p}\right\|_{\Lp(\mr^{d})} \notag \\ 
		&\quad \lesssim 2^{k(-d/p-\Re \alpha+\epsilon)}\|f\|_{\Lp(\mr^{d})} 
	\end{align}
	for $d=2$ and $p\in [4,\infty)$, or $d\geq 3$ and $p\in [2(d+1)/(d-1),\infty)$. Assume for the moment that this inequality is true. By \eqref{eq-local-Lp} and a simple scaling argument, the following inequality holds uniformly for $l\in \mathbb{Z}$,
	\begin{align}\label{eq-scal}
		&\left\|\left(\int_{2^{l}}^{2^{l+1}}\left|A_{t,k}^{\alpha}f(\cdot)\right|^{p}\frac{dt}{t}\right)^{1/p}\right\|_{\Lp(\mr^{d})}+2^{-k}\left\|\left(\int_{2^{l}}^{2^{l+1}}\left|t\frac{\partial }{\partial t }A_{t,k}^{\alpha}f(\cdot)\right|^{p}\frac{dt}{t}\right)^{1/p}\right\|_{\Lp(\mr^{d})} \notag \\ 
		&\quad \lesssim 2^{k(-d/p-\Re \alpha+\epsilon)}\|f\|_{\Lp(\mr^{d})}. 
	\end{align}
	Applying \eqref{eq-scal} and the Littlewood-Paley inequality, we see that
	\begin{align*}
		\left\|\left(\int_{0}^{\infty}\left|A_{t,k}^{\alpha}f(\cdot)\right|^{p}\frac{dt}{t}\right)^{1/p}\right\|_{\Lp(\mr^{d})}^{p}&=\sum_{l=-\infty}^{\infty}\left\|\left(\int_{2^{l}}^{2^{l+1}}\left|A_{t,k}^{\alpha}\Big(\sum_{|j-k+l|\leq 2}P_{j}f\Big)(\cdot)\right|^{p}\frac{dt}{t}\right)^{1/p}\right\|_{\Lp(\mr^{d})}^{p} \\
		&\lesssim 2^{pk(-d/p-\Re \alpha+\epsilon)}\sum_{l=-\infty}^{\infty}\left\|\sum_{|j-k+l|\leq 2}P_{j}f\right\|_{\Lp(\mr^{d})}^{p} \\
		&\lesssim 2^{pk(-d/p-\Re \alpha+\epsilon)}\left\|f\right\|_{\Lp(\mr^{d})}^{p}.
	\end{align*}
	Similarly, 
	\begin{align*}
		2^{-k}\left\|\left(\int_{0}^{\infty}\left|t\frac{\partial }{\partial t }A_{t,k}^{\alpha}f(\cdot)\right|^{p}\frac{dt}{t}\right)^{1/p}\right\|_{\Lp(\mr^{d})}\lesssim 2^{k(-d/p-\Re \alpha+\epsilon)}\|f\|_{\Lp(\mr^{d})}.
	\end{align*}
	Thus, the claim holds. 
	
	It remains to show \eqref{eq-local-Lp}. 
	Let
	\begin{align*}
		\mathcal{F}_{k}^{\pm}f(x,t):=\int_{\mr^{d}}e^{2\pi i (x\cdot \xi\pm t|\xi|)}\hat{f}(\xi)a^{\pm}(|t\xi|)\psi_{k}(t\xi)d\xi,
	\end{align*}
where $a^{\pm}$ are standard symbols of order $0$. Suppose $d=2$ and $p\in [4,\infty)$, or $d\geq 3$ and $p\in [2(d+1)/(d-1),\infty)$. Then the following local smoothing estimate holds
	\begin{align}\label{eq-fourier-Lp}
		\|\mathcal{F}_{k}^{\pm}f\|_{L^{p}(\mathbb{R}^{d}\times [1,2])}\lesssim 2^{k[(d-1)/2-d/p+\epsilon]}\|f\|_{L^{p}(\mathbb{R}^{d})}
	\end{align}
for any $\epsilon>0$, see \cite{Beltran21localsmooth,Bou15AnnMath,Guth2020Ann} for further details. 
Now, we use \eqref{eq-fourier-Lp} to show \eqref{eq-local-Lp}. Recall that
	\begin{align}\label{eq-def-Atk-alpha}
		\widehat{A_{t,k}^{\alpha}f}(\xi)=\widehat{f}(\xi)m^{\alpha}(t\xi)\psi_{k}(t\xi). 
	\end{align}
	From \eqref{al-Bessel}, $A_{t,k}^{\alpha}f$ (resp. $\frac{\partial}{\partial t}A_{t,k}^{\alpha}f$) is a linear combination of $2^{k[(1-d)/2-\Re\alpha]}\mathcal{F}_{k}^{\pm}f(\cdot,t)$ (resp. $2^{k[(3-d)/2-\Re\alpha]}\mathcal{F}_{k}^{\pm}f(\cdot,t)$) for some suitable symbols. Hence, by \eqref{eq-fourier-Lp}, we have
	\begin{align*}
		&\left\|\left(\int_{1}^{2}\left|A_{t,k}^{\alpha}f(\cdot)\right|^{p}\frac{dt}{t}\right)^{1/p}\right\|_{\Lp(\mr^{d})}+2^{-k}\left\|\left(\int_{1}^{2}\left|t\frac{\partial }{\partial t }A_{t,k}^{\alpha}f(\cdot)\right|^{p}\frac{dt}{t}\right)^{1/p}\right\|_{\Lp(\mr^{d})}  \\ 
		&\quad \lesssim 2^{k(-d/p-\Re \alpha+\epsilon)}\|f\|_{\Lp(\mr^{d})}.
	\end{align*}
	
Similarly, the proofs of (ii) and (iii) will be finished by interpolating with \eqref{eq-fourier-Lp} and the $L^{2}$ bounds
\begin{align*}
	\|\mathcal{F}_{k}^{\pm}f\|_{L^{2}(\mathbb{R}^{d}\times [1,2])}\lesssim\|f\|_{L^{2}(\mathbb{R}^{d})}, 
\end{align*}
we omit the details.
\end{proof}

Next, we obtain some estimates of square functions, which will be crucial for estimating $\|V_{2}^{\text{\rm sh}}(\mathcal{A}^{\alpha}_{k}f)\|_{\Lp(\mr^{d})}$.
\begin{lem}\label{lem-square-local}
	Let $k\geq 1$ and
	\begin{align*}
		\mathrm{II}:=\left\|\left(\int_{1}^{2}\left|A_{t,k}^{\alpha}f(\cdot)\right|^{2}dt\right)^{1/2}\right\|_{\Lp(\mr^{d})}+2^{-k}\left\|\left(\int_{1}^{2}\left|t\frac{\partial}{\partial t}A_{t,k}^{\alpha}f(\cdot)\right|^{2}dt\right)^{1/2}\right\|_{\Lp(\mr^{d})}. 
	\end{align*}
    For any $\epsilon>0$,
	\begin{enumerate}
		\item if $d=2$ and $p\in [4,\infty)$, or $d\geq3$ and $p\in[2(d+1)/(d-1),\infty)$, then 
		\begin{align}
			\mathrm{II} \lesssim 2^{k(-d/p-\Re \alpha+\epsilon)}\|f\|_{\Lp(\mr^{d})}\label{eq-square};
		\end{align}
	    \item if $p\in[2,4]$, then 
	        \begin{align*}
	        	\mathrm{II}\lesssim2^{k(-1/2-\Re \alpha+\epsilon)}\|f\|_{\Lp(\mr^{2})};
	        \end{align*}
        \item if $d\geq3,p\in[2,2(d+1)/(d-1)]$, then 
        \begin{align*}
        	\mathrm{II} \lesssim 2^{k[(1-d)(1/4+1/2p)-\Re \alpha+\epsilon]}\|f\|_{\Lp(\mr^{d})}.
        \end{align*}
	\end{enumerate}
\end{lem}
\begin{proof}
	It is sufficient to prove part (i), since (ii) and (iii) can be obtained by interpolation. Define the operators $W_{k, t}^{ \pm}$ by
	$$
	\widehat{W_{k, t}^{ \pm} f}(\xi):=a_k(|\xi|) e^{ \pm 2\pi it|\xi|} \widehat{f}(\xi),
	$$
	where $a_k$ is a standard symbol of order $0$ supported in $\left(2^{k-3}, 2^{k+3}\right)$. Let $I$ be a compact interval in $(0,\infty)$. For $d=2$ and $p\in [4,\infty)$, or $d\geq 3$ and $p\in [2(d+1)/(d-1),\infty)$, the following inequality was established in \cite[p.~6735]{Jones08Trans},
	\begin{align}\label{eq-mix-norm}
		\left\|\left(\int_I\left|W_{k, t}^{ \pm} f(\cdot)\right|^2 d t\right)^{1 / 2}\right\|_{\Lp(\mathbb{R}^{d})} \lesssim 2^{k[(d-1)/2-d/p+\epsilon]}\|f\|_{\Lp(\mathbb{R}^{d})}.
	\end{align}

We now show how to get \eqref{eq-square} by applying \eqref{eq-mix-norm}. By \eqref{al-Bessel} and \eqref{eq-def-Atk-alpha}, we deduce that
    \begin{align*}
    	m^{\alpha}(t\xi)\psi_{k}(t\xi)=2^{k((1-d)/2-\Re \alpha)}\left[e^{-2\pi it|\xi|}h^{\alpha,-}(t|\xi|)+e^{2\pi it|\xi|}h^{\alpha,+}(t|\xi|)\right]\psi_{k}(t\xi),
    \end{align*}
where $h^{\alpha,\pm}$ are symbols of order $0$. Note that \eqref{eq-mix-norm} cannot be applied directly since $h^{\alpha,\pm}$ depend on $t$. So, we use the
Fourier expansion to express $h^{\alpha,\pm}$ in terms of an infinite sum of functions and apply \eqref{eq-mix-norm} to each part. To be precise, let
    \begin{align*}
    	\widehat{H_{k, t}^{\alpha,\pm} f}(\xi):=e^{\pm2\pi it|\xi|}h^{\alpha,\pm}( t|\xi|)\psi_{k}(t\xi)\hat{f}(\xi).
    \end{align*}
Now, the proof of \eqref{eq-square} is reduced to showing
  \begin{align*}
  	\left\|\left(\int_{1}^{2}\left|H_{k, t}^{\alpha,\pm} f(\cdot)\right|^2 d t\right)^{1 / 2}\right\|_{\Lp(\mathbb{R}^{d})}\lesssim 2^{k[(d-1)/2-d/p+\epsilon]}\|f\|_{\Lp(\mathbb{R}^{d})}.
  \end{align*}
By similarity, we only estimate 
\begin{align*}
	\left\|\left(\int_{1}^{3/2}\left|H_{k, t}^{\alpha,\pm} f(\cdot)\right|^2 d t\right)^{1 / 2}\right\|_{\Lp(\mathbb{R}^{d})}.
\end{align*}
Choosing $\eta\in \mathcal{S}(\mathbb{R})$ such that $\eta(t)=1$ if $t\in[1,3/2]$ and $\eta(t)=0$ if $t\notin (3/4,7/4)$. Then
    \begin{align*}
    	\int_{1}^{3/2}|H_{k, t}^{\alpha,\pm}f(x)|^{2}dt=\int_{1}^{3/2}\left|\int_{\mathbb{R}^{d}}e^{2\pi ix\cdot \xi}e^{\pm2\pi it|\xi|}h^{\alpha,\pm}( t|\xi|)\psi_{k}(t\xi)\eta(t)\big[\widetilde{\psi}_{k}(\xi)\big]^{2}\hat{f}(\xi)d\xi\right|^{2}dt,
    \end{align*}
    where $\widetilde{\psi_{k}}\in \mathcal{S}(\mathbb{R}^{d})$ and $\widetilde{\psi_{k}}(\xi)=1$ if $|\xi|\in [2^{k+1}/7,2^{k+3}/3]$ and $\widetilde{\psi_{k}}(\xi)=0$ if $|\xi| \notin (2^{k-2},2^{k+2})$. Given $\xi \in \mathbb{R}^{d}$, define
    \begin{align*}
    	h^{\alpha,\pm}_{k,\xi}(t):=h^{\alpha,\pm}(t|\xi|)\psi_{k}(t\xi)\eta(t).
    \end{align*}
    By the Fourier expansion, for any $t\in (3/4,7/4)$,
    \begin{align*}
    	h^{\alpha,\pm}_{k,\xi}(t)=\sum_{l\in \mathbb{Z}}\widehat{h^{\alpha,\pm}_{k,\xi}}(l)e^{2\pi ilt}.
    \end{align*}
    Let
    \begin{align*}
    	\widehat{F_{k,l}^{\alpha,\pm}f}(\xi):=\widehat{h^{\alpha,\pm}_{k,\xi}}(l)\widetilde{\psi_{k}}(\xi)\hat{f}(\xi)
    \end{align*} 
    and
    \begin{align*}
    	\widehat{H_{k,l,t}^{\alpha,\pm} f}(\xi):=e^{\pm2\pi it|\xi|}\widehat{F_{k,l}^{\alpha,\pm}f}(\xi)\widetilde{\psi_{k}}(\xi).
    \end{align*}
Hence, by \eqref{eq-mix-norm}, we have
\begin{align*}
	\left\|\left(\int_{1}^{3/2}\left|H_{k, t}^{\alpha,\pm} f(\cdot)\right|^2 d t\right)^{1 / 2}\right\|_{\Lp(\mathbb{R}^{d})}&\leq \sum_{l\in \mathbb{Z}}\left\|\left(\int_{1}^{3/2}\left|H_{k,l,t}^{\alpha,\pm} f(\cdot)\right|^2 d t\right)^{1 / 2}\right\|_{\Lp(\mathbb{R}^{d})}\\
	&\lesssim 2^{k[(d-1)/2-d/p+\epsilon]}\sum_{l\in \mathbb{Z}}\left\| F_{k,l}^{\alpha,\pm} f\right\|_{\Lp(\mathbb{R}^{d})}.
\end{align*}

It remains to show
\begin{align}\label{eq-sum-fourier}
	\sum_{l\in\mathbb{Z}}\left\| F_{k,l}^{\alpha,\pm} f\right\|_{\Lp(\mathbb{R}^{d})}\lesssim\|f\|_{\Lp(\mathbb{R}^{d})}.
\end{align}
By integration by parts, it is not difficult to see that the kernel $K_{k,l}^{\alpha,\pm}$ of the operator $F_{k,l}^{\alpha,\pm}$ have the following estimate,
\begin{align*}
	|K_{k,l}^{\alpha,\pm}(x)|\leq \frac{C_{N}}{[(1+|x|)(1+|l|)]^{N}}
\end{align*}
for any $N\in \mathbb{N}$. Then by Young's convolution inequality, we get \eqref{eq-sum-fourier}.
\end{proof}
By Lemmas \ref{lem-local} and \ref{lem-square-local}, we deduce the following global version of Lemma \ref{lem-square-local}.
\begin{lem}\label{lem-square-global}
	Let $k\geq 1$ and
	\begin{align*}
		\mathrm{III}:=\left\|\left(\int_{0}^{\infty}\left|A_{t,k}^{\alpha}f(\cdot)\right|^{2}\frac{dt}{t}\right)^{1/2}\right\|_{\Lp(\mr^{d})}+2^{-k}\left\|\left(\int_{0}^{\infty}\left|t\frac{\partial}{\partial t}A_{t,k}^{\alpha}f(\cdot)\right|^{2}\frac{dt}{t}\right)^{1/2}\right\|_{\Lp(\mr^{d})}.
	\end{align*}
    For any $\epsilon>0$, 
		\begin{enumerate}
		\item if $d=2$ and $p\in [4,\infty)$, or $d\geq3$ and $p\in[2(d+1)/(d-1),\infty)$, then 
		\begin{align*}
			\mathrm{III} \lesssim 2^{k(-d/p-\Re \alpha+\epsilon)}\|f\|_{\Lp(\mr^{d})};
		\end{align*}
		\item if $p\in [2,4]$, 
		\begin{align*}
			\mathrm{III} \lesssim 2^{k(-1/2-\Re \alpha+\epsilon)}\|f\|_{\Lp(\mr^{2})};
		\end{align*}
	 \item if $d\geq3$ and $p\in[2,2(d+1)/(d-1)]$, then 
	\begin{align*}
		\mathrm{III} \lesssim 2^{k[(1-d)(1/4+1/2p)-\Re \alpha+\epsilon]}\|f\|_{\Lp(\mr^{d})}.
	\end{align*}
	\end{enumerate}
\end{lem}
\begin{proof}
By similarity, we only prove
	\begin{align*}
		\left\|\left(\int_{0}^{\infty}\left|A_{t,k}^{\alpha}f(\cdot)\right|^{2}\frac{dt}{t}\right)^{1/2}\right\|_{\Lp(\mr^{d})}\lesssim 2^{k(-d/p-\Re \alpha+\epsilon)}\|f\|_{\Lp(\mr^{d})}
	\end{align*}
for $d=2$ and $p\in [4,\infty)$, or $d\geq3$ and $p\in[2(d+1)/(d-1),\infty)$. After a change of variable, it suffices to estimate
    \begin{align*}
    	\left\|\left(\sum_{j\in\mathbb{Z}}\int_{1}^{2}\left|A_{2^{j+k}t,k}^{\alpha}f(\cdot)\right|^{2}\frac{dt}{t}\right)^{1/2}\right\|_{\Lp(\mr^{d})}.
    \end{align*}
Note that
\begin{align*}
	\widehat{A_{2^{j+k}t,k}^{\alpha}f}(\xi)=\widehat{f}(\xi)m^{\alpha}(2^{j+k}t\xi)\psi_{k}(2^{j+k}t\xi)=:\widehat{f}(\xi)m^{\alpha,k}_{t}(2^{j}\xi),
\end{align*}
where
\begin{align*}
	m^{\alpha,k}_{t}(\xi):=m^{\alpha}(2^{k}t\xi)\psi_{k}(2^{k}t\xi).
\end{align*}
Combining Lemma \ref{lem-square-local} (i) and a standard scaling argument, we have
 \begin{align*}
	\left\|\left(\int_{1}^{2}\left|A_{2^{j+k}t,k}^{\alpha}f(\cdot)\right|^{2}\frac{dt}{t}\right)^{1/2}\right\|_{\Lp(\mr^{d})}\lesssim 2^{k(-d/p-\Re \alpha+\epsilon)}\|f\|_{\Lp(\mr^{d})},
\end{align*}
which further implies
\begin{align}\label{eq-supj-p}
	\sup_{j\in \mathbb{Z}}\left\|\left(\int_{1}^{2}\left|A_{2^{j+k}t,k}^{\alpha}f(\cdot)\right|^{2}dt\right)^{1/2}\right\|_{\Lp(\mr^{d})}\lesssim2^{k(-d/p-\Re \alpha+\epsilon)}\|f\|_{\Lp(\mr^{d})}.
\end{align}

By Plancherel's theorem, for $p\geq 2d/(d-1)$,
\begin{align}
	\sup_{j\in \mathbb{Z}}\left\|\left(\int_{1}^{2}\left|A_{2^{j+k}t,k}^{\alpha}f(\cdot)\right|^{2}dt\right)^{1/2}\right\|_{L^{2}(\mr^{d})}&\lesssim2^{k[(1-d)/2-\Re \alpha]}\|f\|_{L^{2}(\mr^{d})}\notag \\
	&\lesssim 2^{k(-d/p-\Re \alpha+\epsilon)}\|f\|_{\Lp(\mr^{d})}. \label{eq-supj-2}
\end{align}
Combining Lemma \ref{lem-local}, \eqref{eq-supj-p}, \eqref{eq-supj-2} and the fact
\begin{align*}
	\sup _{s \in [1,2]}\left|\partial_{\xi}^{\tau} m_{s}^{\alpha,k}(\xi)\right| \lesssim 2^{k(d+1)} \quad \text { for each } 0 \leq|\tau| \leq d+1,
\end{align*}
we get
\begin{align*}
	\left\|\left(\sum_{j\in\mathbb{Z}}\int_{1}^{2}\left|A_{2^{j+k}t,k}^{\alpha}f(\cdot)\right|^{2}\frac{dt}{t}\right)^{1/2}\right\|_{\Lp(\mr^{d})}\lesssim2^{k(-d/p-\Re \alpha+\epsilon)}\|f\|_{\Lp(\mr^{d})}.
\end{align*}
This completes the proof.
\end{proof}
Now, we proceed to prove Theorem \ref{thm-main} (ii).
\begin{proof}[Proof of Theorem \ref{thm-main} (ii)] (a$_{4}^{\prime}$) is trivial. (a$_{1}$) follows from Proposition \ref{prop-long-varia} and Lemma \ref{lem-variation-S2}. It remains to show (a$_{2}^{\prime}$) and (a$_{3}^{\prime}$). 
	
	We first prove (a$_{3}^{\prime}$). From Lemma \ref{lem-smoothing} (i), Lemma \ref{lem-square-global} (i) and Lemma \ref{lem-variation-norm}, it follows that
	\begin{align*}
		\|V_{p}^{\text{\rm sh}}(\mathcal{A}_{k}^{\alpha}f)\|_{\Lp(\mr^{2})}&\lesssim \left(\int_{\mr^{2}}\int_{0}^{\infty}\left|A_{t,k}^{\alpha}f(x)\right|^{p}\frac{dt}{t}dx\right)^{1/(p p^{\prime})}\times \left(\int_{\mr^{2}}\int_{0}^{\infty}\left|t \frac{\partial}{\partial t}A_{t,k}^{\alpha}f(x)\right|^{p}\frac{dt}{t}dx\right)^{1/p^{2}} \notag \\
		&\lesssim 2^{k(-\Re\alpha-1/p+\epsilon)}\|f\|_{L^{p}(\mathbb{R}^{2})}
	\end{align*}
for $k\geq 1$ and $p\in [4,\infty)$. Combining this with Lemma \ref{lem-varia-decom}, Proposition \ref{prop-long-varia} and Corollary \ref{cor-short-var}, we obtain
	\begin{align*}
		\|V_{p}(\mathcal{A}^{\alpha}f)\|_{\Lp(\mr^{2})}\lesssim \|f\|_{\Lp(\mr^{2})}
	\end{align*}
for $\Re \alpha>-1/p$. Since $\|\cdot\|_{v_{q}}\leq \|\cdot\|_{v_{p}}$ for $q\geq p$, we further deduce
	\begin{align*}
		\|V_{q}(\mathcal{A}^{\alpha}f)\|_{\Lp(\mr^{2})}\lesssim \|f\|_{\Lp(\mr^{2})}
	\end{align*}
under the condition $p\in [4,\infty), \Re \alpha>-1/p$ and $q\geq p$.

It remains to consider the case when $p\in [4,\infty), \Re \alpha>-1/p$ and $q\in (2,p]$.  By Lemma \ref{lem-square-global} (i) and Lemma \ref{lem-variation-norm}, for $k\geq 1$, we have
\begin{align*}
	&\|V_{2}^{\text{\rm sh}}(\mathcal{A}_{k}^{\alpha}f)\|_{\Lp(\mr^{2})}\notag\\
	&\quad \lesssim 2^{k/2}\left\|\left(\int_{0}^{\infty}\left|A_{t,k}^{\alpha}f(\cdot)\right|^{2}\frac{dt}{t}\right)^{1/2}\right\|_{\Lp(\mr^{2})} +2^{-k/2}\left\|\left(\int_{0}^{\infty}\left|t\frac{\partial}{\partial t}A_{t,k}^{\alpha}f(\cdot)\right|^{2}\frac{dt}{t}\right)^{1/2}\right\|_{\Lp(\mr^{2})}\notag\\
	&\quad \lesssim 2^{k(1/2-2/p-\Re \alpha+\epsilon)}\|f\|_{\Lp(\mr^{2})}.
\end{align*}
By interpolation, 
\begin{align*}
	\|V_{q}^{\text{\rm sh}}(\mathcal{A}_{k}^{\alpha}f)\|_{\Lp(\mr^{2})}\lesssim 2^{k(1/q-2/p-\Re \alpha+\epsilon)}\|f\|_{\Lp(\mr^{2})}
\end{align*}
holds for $q\in (2,p]$ and $k\geq 1$. Combining this with Lemma \ref{lem-varia-decom}, Proposition \ref{prop-long-varia} and Corollary \ref{cor-short-var}, we see
\begin{align*}
	\|V_{q}(\mathcal{A}^{\alpha}f)\|_{\Lp(\mr^{2})}\lesssim \|f\|_{\Lp(\mr^{2})}
\end{align*}
when $1/q<2/p+\Re \alpha$. Thus, we complete the proof of (a$_{3}^{\prime}$).

For (a$_{2}^{\prime}$), by Lemma \ref{lem-smoothing} (ii), Lemma \ref{lem-square-global} (ii) and Lemma \ref{lem-variation-norm}, we have
\begin{align*}
	\|V_{p}^{\text{\rm sh}}(\mathcal{A}_{k}^{\alpha}f)\|_{\Lp(\mr^{2})}\lesssim 2^{k(-\Re\alpha+1/p-1/2+\epsilon)}\|f\|_{L^{p}(\mathbb{R}^{2})}
\end{align*}
and
	\begin{align*}
		\|V_{2}^{\text{\rm sh}}(\mathcal{A}_{k}^{\alpha}f)\|_{\Lp(\mr^{2})}\lesssim 2^{k(-\Re \alpha+\epsilon)}\|f\|_{\Lp(\mathbb{R}^{2})}
	\end{align*}
for $k\geq 1$ and $p\in [2,4]$. Using interpolation again,
	\begin{align*}
		\|V_{q}^{\text{\rm sh}}(\mathcal{A}_{k}^{\alpha}f)\|_{\Lp(\mr^{2})}\lesssim 2^{k(1/q-1/2-\Re \alpha+\epsilon)}\|f\|_{\Lp(\mr^{2})}
	\end{align*}
holds for $q\in [2,p]$. Thus, we see that $V_{q}(\mathcal{A}^{\alpha})$ is bounded on $\Lp(\mathbb{R}^{2})$ when $p\in [2,4]$, $\Re \alpha>1/p-1/2$ and $1/q<1/2+\Re \alpha$.
\end{proof}

\begin{proof}[Proof of Theorem \ref{thm-main} (iii)] By Lemma \ref{lem-smoothing} (i), Lemma \ref{lem-square-global} (i) and Lemma \ref{lem-variation-norm}, we have
	\begin{align*}
		\|V_{p}^{\text{\rm sh}}(\mathcal{A}_{k}^{\alpha}f)\|_{\Lp(\mr^{d})}\lesssim 2^{k[(1-d)/p-\Re\alpha+\epsilon]}\|f\|_{L^{p}(\mathbb{R}^{d})}
	\end{align*}
and
    \begin{align*}
    	\|V_{2}^{\text{\rm sh}}(\mathcal{A}_{k}^{\alpha}f)\|_{\Lp(\mr^{d})}\lesssim 2^{k(1/2-d/p-\Re\alpha+\epsilon)}\|f\|_{L^{p}(\mathbb{R}^{d})}
    \end{align*}
for $d\geq 3$ and $p\in[2(d+1)/(d-1),\infty)$.
From Lemma \ref{lem-smoothing} (iii), Lemma \ref{lem-square-global} (iii) and Lemma \ref{lem-variation-norm}, it follows that
    \begin{align*}
	\|V_{p}^{\text{\rm sh}}(\mathcal{A}_{k}^{\alpha}f)\|_{\Lp(\mr^{d})}\lesssim 2^{k[(1-d)/4+(3-d)/(2p)-\Re\alpha+\epsilon]}\|f\|_{L^{p}(\mathbb{R}^{d})}
    \end{align*}
and
    \begin{align*}
	\|V_{2}^{\text{\rm sh}}(\mathcal{A}_{k}^{\alpha}f)\|_{\Lp(\mr^{d})}\lesssim 2^{k[(3-d)/4+(1-d)/(2p)-\Re\alpha+\epsilon]}\|f\|_{L^{p}(\mathbb{R}^{d})}
    \end{align*}
for $d\geq 3$ and $p\in[2,2(d+1)/(d-1)]$. Now one can argue as in the proof of Theorem \ref{thm-main} (ii) to finish the proof, we omit the details. 
\end{proof}

\section{Necessary conditions}\label{sec-nec}
We prove Theorem \ref{thm-main} (i) in this section. By Theorem 1.1 (i) in \cite{Liu2023arXiv} for $p\geq2$ and the comment in \cite[p.~519]{Stein93} for $1<p\leq 2$, we immediately get Theorem \ref{thm-main} (i) when $q=\infty$ and Theorem \ref{thm-main} (i) (a$_{1}$). So, it suffices to show Theorem \ref{thm-main} (i) (a$_{2}$), (a$_{3}$) when $q\in (2,\infty)$ and (a$_{4}$). We first establish the following lemma.
\begin{lem}\label{lem-flambda}
	Suppose $\chi \in C_{c}^{\infty}(\mathbb{R})$ and is nonnegative, satisfying $\supp \chi \subset [1 / 2,3 / 2]$ and $\chi \equiv 1$ on $[3 / 4,5 / 4]$. For $\alpha\in \mathbb{C}$ and $\lambda>0$, let
	$$
	\widehat{f_\lambda^{\alpha}}(\xi):=\chi\left(\lambda^{-1}|\xi|\right)|\xi|^{i\Im \alpha} e^{\pi i|\xi|^2/\lambda}.
	$$
    Then for $p\in[1,\infty)$,
    \begin{align*}
    	\|f_{\lambda}\|_{\Lp(\mr^{d})}\lesssim \lambda^{d/2}.
    \end{align*}
\end{lem}
\begin{proof}
	After a change of variable, we rewrite
	$$
	f_{\lambda}^{\alpha}(x)=\lambda^{d+i\Im \alpha} \int_{\mr^{d}} \chi(|\xi|)|\xi|^{i \Im \alpha} e^{2\pi i \lambda(|\xi|^{2}/2+x\cdot \xi)} d \xi .
	$$
	By \cite[Proposition 4, p.~341]{Stein93}, we get
	\begin{align*}
		|f_{\lambda}^{\alpha}(x)|\lesssim|\lambda x|^{-N}
	\end{align*}
for $|x|\geq 2$ and $N\in \mathbb{N}$. For $|x|\leq 2$, by the method of stationary phase (see \cite[Proposition 6, p.~344]{Stein93}), we have
	\begin{align*}
		\left|f_{\lambda}^{\alpha}(x)\right| \lesssim \lambda^{d/2}.
	\end{align*}
Combining these estimates for $f_{\lambda}(x)$, the conclusion follows.
\end{proof}

\begin{prop}\label{prop-q-range}
	Let $\alpha\in \mathbb{C}$, $q\in (2,\infty)$ and $p\in (1,\infty)$. If $V_{q}(\mathcal{A}^{\alpha})$ is bounded on $\Lp(\mr^{d})$, then
    \begin{align*}
    	\frac{1}{q}\leq \Re \alpha+\frac{d}{p}.
    \end{align*}
\end{prop}
\begin{proof}
	By a direct calculation, we get
\begin{align*}
	A_{t}^{\alpha}f_{\lambda}^{\alpha}(x)
=\pi^{-\alpha+1}\lambda^{d/2+1-\Re \alpha}t^{-d/2- \alpha+1}\int_{0}^{\infty}J_{d/2+\alpha-1}(2\pi t\lambda r) e^{\pi i\lambda r^{2}}\chi(r)r^{d/2-\Re \alpha}\vartheta(\lambda r|x|)dr,
\end{align*}
where 
\begin{align*}
	\vartheta(\lambda r|x|):=\int_{\mathbb{S}^{d-1}}e^{2\pi i |x|\lambda r \vec{e}_{1}\cdot \theta}d\sigma(\theta).
\end{align*}
Suppose $c_{0}$ is small enough and $|x|\leq c_{0}\lambda^{-1}$. By \eqref{al-Bessel}, we have
\begin{align*}
	J_{d/2+1-\alpha}(r)=r^{-1/2}(e^{ir}b_{0,\alpha}+e^{-ir}d_{0,\alpha})+a_{1,\alpha}(r),\quad r\geq 1,
\end{align*}
where 
\begin{align*}
	|a_{1,\alpha}(r)|\lesssim r^{-3/2}.
\end{align*}
Combining this and the fact $|\vartheta(\lambda r|x|)|\lesssim 1$ when $|\lambda rx|\leq c_{0}$, we obtain
\begin{align*}
	&\left|\pi^{-\alpha+1}\lambda^{d/2+1-\Re \alpha}t^{-d/2- \alpha+1}\int_{0}^{\infty}a_{1,\alpha}(2\pi t\lambda r) e^{\pi i\lambda r^{2}}\chi(r)r^{d/2-\Re \alpha}\vartheta(\lambda r|x|)dr\right|\\
	&\quad \lesssim \lambda^{(d-1)/2-\Re \alpha}t^{-(d+1)/2-\Re \alpha}.
\end{align*}
Since there is no critical point for the phase function $tr+r^{2}/2$, we deduce
\begin{align*}
	&\left|\pi^{-\alpha+1}\lambda^{d/2+1-\Re \alpha}t^{-d/2- \alpha+1}\int_{0}^{\infty}(2\pi t\lambda r)^{-1/2}b_{0,\alpha}e^{2\pi i\lambda( tr+r^{2}/2)}\chi(r)r^{d/2-\Re \alpha}\vartheta(\lambda r|x|)dr\right| \\
	&\quad \lesssim \lambda^{-N}t^{(d-1)/2-\Re \alpha}.
\end{align*}

Finally, we estimate the main term
\begin{align*}
	d_{0,\alpha}\pi^{-\alpha+1}(2\pi)^{-1/2}\lambda^{(d+1)/2-\Re \alpha}t^{(1-d)/2- \alpha}\int_{0}^{\infty}e^{2\pi i\lambda(- tr+r^{2}/2)}\chi(r)r^{(d-1)/2-\Re \alpha}\vartheta(\lambda r|x|)dr.
\end{align*}
Note the phase has a nondegenerate critical point at $r=t$. For $t\approx 1$ and $|x|\leq c_{0}\lambda^{-1}$, by the method of stationary phase, we have
\begin{align*}
	&\lambda^{(d+1)/2-\Re \alpha}t^{(1-d)/2-\alpha}\int_{0}^{\infty}e^{2\pi i\lambda(- tr+r^{2}/2)}\chi\left(r\right)r^{(d-1)/2-\Re \alpha}\vartheta(\lambda r|x|)dr \\
	&\quad =\lambda^{d/2-\Re \alpha}t^{-(\alpha+\Re \alpha)}e^{-\pi i\lambda t^{2}}\chi(t)\vartheta(\lambda |x|t)+\mathcal{O}(\lambda^{(d-2)/2-\Re \alpha}). 
\end{align*}
For $1\leq n\leq \lambda/100$, we choose
\begin{align*}
	t_{n}:=\sqrt{1+\frac{n}{\lambda}}.
\end{align*}
Then
\begin{align*}
	\left|e^{-\pi i\lambda t_{n+1}^{2}}-e^{-\pi i\lambda t_{n}^{2}}\right|=2 \text{ and } |t_{n+1}-t_{n}|\lesssim \frac{1}{\lambda}.
\end{align*}
Therefore,
\begin{align*}
	\left|A_{t_{n+1}}^{\alpha}f_{\lambda}^{\alpha}(x)-A_{t_{n}}^{\alpha}f_{\lambda}^{\alpha}(x)\right|&\geq \lambda^{d/2-\Re \alpha}(C_{1}-C_{2}\lambda^{-1})
\end{align*}
holds for some $C_{1},C_{2}>0$, which further deduces
\begin{align*}
	\left(\sum_{1 \leq n<\lambda/100}\left|A_{t_{n+1}}^{\alpha}f_{\lambda}^{\alpha}(x)-A_{t_{n}}^{\alpha}f_{\lambda}^{\alpha}(x)\right|^q\right)^{1 / q} \geq C\lambda^{d/2-\Re \alpha+1/q}
\end{align*}
when $\lambda$ is large enough. By the assumption and Lemma \ref{lem-flambda}, we get
\begin{align*}
	\lambda^{d/2-\Re \alpha+1/q-d/p}\lesssim \lambda^{d/2},
\end{align*}
this yields
\begin{align*}
	\frac{1}{q}\leq \Re \alpha+\frac{d}{p}.
\end{align*}
The proof is completed.
\end{proof}
To get the following conclusion, we focus on $|x|\approx 3$ instead of $|x|\leq c_{0}\lambda$.
\begin{prop}\label{prop-q-range2}
	Let $\alpha\in \mathbb{C}$, $q\in (2,\infty)$ and $p\in (1,\infty)$. If $V_{q}(\mathcal{A}^{\alpha})$ is bounded on $\Lp(\mr^{d})$, then
	\begin{align*}
		\frac{1}{q}\leq \frac{d-1}{2}+\Re \alpha.
	\end{align*}
\end{prop}
\begin{proof}
	We consider $|x|\approx 3$ and $t\approx 2$. As in Proposition \ref{prop-q-range}, we have
	\begin{align*}
		A_{t}^{\alpha}f_{\lambda}^{\alpha}(x)&=\pi^{-\alpha+1}\lambda^{d/2+1-\Re \alpha}t^{-d/2-\alpha+1} \int_{0}^{\infty}\int_{\mathbb{S}^{d-1}}e^{2\pi i x\cdot \lambda r\theta}d\sigma(\theta)J_{d/2+\alpha-1}(2\pi t\lambda r) \\
		&\quad  e^{\pi i\lambda r^{2}}\chi\left(r\right)r^{d/2-\Re \alpha}dr. 
	\end{align*}
It follows from \eqref{al-Bessel} that
\begin{align*}
	J_{(d-2)/2}(r)=r^{-1/2}e^{ir}c_{1}+r^{-1/2}e^{-ir}e_{1}+a_{2}(r),\quad r\geq 1,
\end{align*}
where
\begin{align*}
	|a_{2}(r)|\lesssim r^{-3/2}, \quad r\geq 1.
\end{align*}
Hence, 
\begin{align*}
	\int_{\mathbb{S}^{d-1}}e^{2\pi i x\cdot \lambda r\theta}d\sigma(\theta)=(2\pi)^{1/2} |\lambda x r|^{(1-d)/2}\left(e^{i|2\pi\lambda x r|}c_{1}+e^{-i|2\pi\lambda x r|}e_{1}\right)+a_{3}(|2\pi\lambda xr|),
\end{align*}
where
\begin{align*}
	|a_{3}(|2\pi\lambda xr|)|\lesssim  |\lambda xr|^{-(d+1)/2}.
\end{align*}
Without loss of generality, we assume $c_{1}=e_{1}=1$.
Similarly, 
\begin{align*}
		J_{d/2+\alpha-1}(2\pi t\lambda r)=(2\pi t\lambda r)^{-1/2}\left(e^{i2\pi t\lambda r}c_{2,\alpha}+e^{-i2\pi t\lambda r}e_{2,\alpha}\right)+a_{4,\alpha}(2\pi t\lambda r),\quad 2\pi t\lambda r\geq 1,
\end{align*}
where
\begin{align*}
	|a_{4,\alpha}(2\pi t\lambda r)|\lesssim \left(t\lambda r\right)^{-3/2}.
\end{align*}
For the sake of simplicity, we may assume $c_{2,\alpha}=e_{2,\alpha}=1$. Thus, we get
\begin{align*}
	&\int_{\mathbb{S}^{d-1}}e^{2\pi i x\cdot \lambda r\theta}d\sigma(\theta)J_{d/2+\alpha-1}(2\pi t\lambda r)\\
	&\quad =(t\lambda r)^{-1/2}|\lambda xr|^{(1-d)/2}\left(e^{i2\pi \lambda r(|x|+t)}+e^{i2\pi \lambda r(|x|-t)}+e^{i2\pi \lambda r(-|x|+t)}+e^{i2\pi \lambda r(-|x|-t)}\right)\\
	&\quad \quad +a_{5,\alpha}(t,\lambda,|x|,r),
\end{align*}
where
\begin{align*}
	|a_{5,\alpha}(t,\lambda,|x|,r)|\lesssim \lambda^{-(d+2)/2}.
\end{align*}
Hence,
\begin{align*}
	A_{t}^{\alpha}f_{\lambda}^{\alpha}(x)&=\pi^{-\alpha+1}\lambda^{d/2+1-\Re \alpha}t^{-d/2-\alpha+1}\int_{0}^{\infty}[(t\lambda r)^{-1/2}|\lambda xr|^{(1-d)/2}\big(e^{i2\pi \lambda r(|x|+t)}+e^{i2\pi \lambda r(|x|-t)}\\
	&\quad  +e^{i2\pi \lambda r(-|x|+t)}+e^{i2\pi \lambda r(-|x|-t)}\big)+a_{5,\alpha}(t,\lambda,|x|,r)]e^{\pi i \lambda r^{2}}\chi\left(r\right)r^{d/2-\Re \alpha}dr \\ 
	&=:\sum_{i=1}^{5}I_{i}(x,t,\lambda).
\end{align*}

For $1\leq n\leq \lambda/1000$ and $11/4\leq |x|\leq 3$, we choose
\begin{align*}
	t_{n}:=|x|-\sqrt{\frac{9}{16}+\frac{n}{\lambda}}.
\end{align*} 
Then
\begin{align*}
	\frac{19}{10}\leq t_{n}\leq \frac{9}{4}.
\end{align*}
For $I_{5}(x,t_{n},\lambda)$, we have
\begin{align*}
	|I_{5}(x,t_{n},\lambda)|&\lesssim \lambda^{d/2+1-\Re \alpha}\int_{0}^{\infty}\lambda^{-(d+2)/2}\chi(r)r^{d/2-\Re \alpha}dr \lesssim \lambda^{-\Re \alpha}.
\end{align*}
Note that the phase functions of $I_{1}(x,t_{n},\lambda), I_{2}(x,t_{n},\lambda)$ and $I_{4}(x,t_{n},\lambda)$ do not have critical points, which implies
\begin{align*}
	|I_{i}(x,t,\lambda)|\lesssim \lambda^{-N}, \quad i=1,2,4.
\end{align*}

It remains to estimate $I_{3}(x,t_{n},\lambda)$. Write
\begin{align*}
	I_{3}(x,t_{n},\lambda)=\pi^{-\alpha+1}\lambda^{1-\Re \alpha}t_{n}^{(1-d)/2- \alpha}|x|^{(1-d)/2}\int_{0}^{\infty}e^{2\pi i\lambda\left(r^{2}/2+r (-|x|+t_{n})\right)}r^{-\Re \alpha}\chi(r)dr.
\end{align*}
Obviously, the phase function of $I_{3}(x,t_{n},\lambda)$ has a nondegenerate critical point at 
\begin{align*}
	r_{n}:=|x|-t_{n}=\sqrt{\frac{9}{16}+\frac{n}{\lambda}}.
\end{align*}
By the method of stationary phase, for $11/4\leq |x|\leq 3$, we obtain
\begin{align*}
	\int_{0}^{\infty}e^{2\pi i\lambda\left(r^{2}/2+r (-|x|+t_{n})\right)}r^{-\Re \alpha}\chi(r)dr=\lambda^{-1/2}e^{2\pi i\lambda\left(r_{n}^{2}/2+r_{n}(-|x|+t_{n})\right)}r_{n}^{-\Re \alpha}\chi(r_{n})+\mathcal{O}(\lambda^{-3/2}).
\end{align*}
From this, it follows that
\begin{align*}
	&I_{3}(x,t_{n},\lambda)\\
	&\quad =\pi^{-\alpha+1}(2\pi)^{(2-d)/2}\lambda^{1/2-\Re \alpha}t_{n}^{(1-d)/2-\alpha}|x|^{(1-d)/2}e^{-\pi i\lambda \cdot r_{n}^{2}}r_{n}^{-\Re \alpha}\chi(r_{n})+\mathcal{O}(\lambda^{-1/2-\Re \alpha}).
\end{align*}
Since
\begin{align*}
	\left|e^{-\pi i\lambda \cdot r_{n+1}^{2}}-e^{-\pi i\lambda \cdot r_{n}^{2}}\right|=2 \text{ and } |r_{n+1}-r_{n}|\lesssim \frac{1}{\lambda},
\end{align*}
we see that
\begin{align*}
	|I_{3}(x,t_{n+1},\lambda)-I_{3}(x,t_{n},\lambda)|\geq \lambda^{1/2-\Re \alpha}(c_{1}-c_{2}\lambda^{-1}),
\end{align*}
for some $c_{1},c_{2}>0$, which implies
\begin{align*}
	\left(\sum_{1\leq n<\lambda/1000}|A_{t_{n+1}}^{\alpha}f_{\lambda}^{\alpha}(x)-A_{t_{n}}^{\alpha}f_{\lambda}^{\alpha}(x)|^{q}\right)^{\frac{1}{q}}\geq C\lambda^{1/2-\Re \alpha+1/q}
\end{align*}
when $\lambda$ is large enough. By the assumption and Lemma \ref{lem-flambda}, we conclude
\begin{align*}
	\frac{1}{q}\leq \frac{d-1}{2}+\Re \alpha.
\end{align*}
This completes the proof. 
\end{proof}

Now, Theorem \ref{thm-main} (i) (a$_{2}$) and (a$_{3}$) follow from Propositions \ref{prop-q-range} and \ref{prop-q-range2}, it remains to show Theorem \ref{thm-main} (i) (a$_{4}$). We first consider $\alpha=1$.
\begin{prop}
	$V_{q}(\mathcal{A}^{1})$ is unbounded on $L^{\infty}(\mr^{d})$ for $q\in (2,\infty)$.
\end{prop}
\begin{proof}
	For any $n\in \mathbb{N}$ and $n\geq 2$, let $\widetilde{c}_{1}:=-1/[(2^{d}-1)m(B(0,1))]$ and
	\begin{align*}
		\widetilde{c}_{n}:=\begin{cases}
			1/m(B(0,1)),& \text{ $n$ is even}; \\
			-1/[2^{d}m(B(0,1))],& \text{ $n$ is odd}.
		\end{cases}
	\end{align*}
Define
	\begin{align*}
		f_{0}(x):=\frac{\chi_{E_{0}}(|x|)}{m(B(0,1))}
		\text{ and }f_{n}(x):=f_{0}(x)+\sum_{j=1}^{n}\widetilde{c}_{j}\chi_{E_{j}}(|x|),
	\end{align*}
where $E_{0}:=[0,1)$ and $E_{j}:=[2^{j-1},2^{j})$ for $1\leq j\leq n$. Let $t_{0}:=1$ and $t_{j}:=2^{j}$ for $1\leq j\leq n$. A direct computation shows $A_{t_{j}}^{1}f_{n}(0)=0$ when $j$ is odd and $A_{t_{j}}^{1}f_{n}(0)\geq 1/2$ when $j$ is even.

Suppose $|x|\leq 1/(d2^{d+3})$ and $1\leq t\leq 2^{n}$, then
\begin{align}\label{eq-error}
	\left|A_{t}^{1}f_{n}(x)-A_{t}^{1}f_{n}(0)\right|\leq \frac{1}{t^{d}}\int_{B(0,t)\Delta B(x,t)}\left|f_{n}(y)\right|dy  \leq 2|x|\cdot \frac{d(t+|x|)^{d-1}}{t^{d}} \leq \frac{1}{8},
\end{align}
where in the second inequality we used the fact
\begin{align*}
	B(0,t)\Delta B(x,t)\subset B(0,t+|x|)\setminus B(0,t-|x|).
\end{align*}
By \eqref{eq-error}, for $0\leq j\leq n-1$, we have
\begin{align*}
	\left|A_{t_{j+1}}^{1}f_{n}(x)-A_{t_{j}}^{1}f_{n}(x)\right|\geq \frac{1}{4},
\end{align*}
from which it follows that 
\begin{align*}
	\left(\sum_{j=0}^{n-1}\left|A_{t_{j+1}}^{1}f_{n}(x)-A_{t_{j}}^{1}f_{n}(x)\right|^{q}\right)^{1/q}\geq \frac{n^{1/q}}{4}.
\end{align*}
Thus, 
\begin{align*}
	\frac{\|V_{q}(\mathcal{A}^{1}f_{n})\|_{L^{\infty}(\mr^{d})}}{\|f_{n}\|_{L^{\infty}(\mr^{d})}}\geq \frac{n^{1/q}\cdot m(B(0,1))}{4},
\end{align*}
which completes the proof.
\end{proof}

If $\Re \alpha> 0$ and $\alpha\neq 1$, the proof is more complicated. 
\begin{prop}
	Let $q\in (2,\infty)$, $\Re \alpha>0$ with $\alpha\neq 1$. Then $V_{q}(\mathcal{A}^{\alpha})$ is unbounded on $L^{\infty}(\mr^{d})$.
\end{prop}
\begin{proof}
	First, we consider $|\alpha-1|\geq 1$. For $j\in \mathbb{N}$, we choose
	\begin{align*}
		t_{j}:=2^{4j}c_{\alpha}^{j},\quad E_{j}:=[2^{4(j-1)}c_{\alpha}^{j-1},2^{4(j-1)+1}c_{\alpha}^{j-1})
	\end{align*}  
	and
	\begin{align*}
		\widetilde{c}_{j}:=\begin{cases}
			d2^{4d}c_{\alpha}^{d}\Gamma(\alpha)/[(2^{d}-1)\sigma(\mathbb{S}^{d-1})], &\text{ if $j$ is odd};\\
			-d\Gamma(\alpha)/[(2^{d}-1)\sigma(\mathbb{S}^{d-1})], &\text{ if $j$ is even},
		\end{cases}
	\end{align*}
where $c_{\alpha}:=|\alpha-1|$. For $n\in \mathbb{N}$, we define 
\begin{align*}
	f_{n}(x):=\sum_{j=1}^{n}\widetilde{c}_{j}\chi_{E_{j}}(|x|).
\end{align*}
If $1\leq j\leq n$ and $j=2k$ for some $k\in \mathbb{N}$, then
\begin{align*}
	A_{t_{j}}^{\alpha}f_{n}(0)&= \frac{d2^{4d}c_{\alpha}^{d}}{t_{j}^{d}(2^{d}-1)}
	\sum_{l=1}^{k}\left[\int_{2^{8l-8}c_{\alpha}^{2l-2}}^{2^{8l-7}c_{\alpha}^{2l-2}}\left(\left(1-\frac{r^{2}}{t_{j}^{2}}\right)^{\alpha-1}-1\right)r^{d-1}dr \right.\\
	&\quad\left. -\frac{1}{c_{\alpha}^{d}2^{4d}}\int_{2^{8l-4}c_{\alpha}^{2l-1}}^{2^{8l-3}c_{\alpha}^{2l-1}}\left(\left(1-\frac{r^{2}}{t_{j}^{2}}\right)^{\alpha-1}-1\right)r^{d-1}dr\right] \\
	&=:\frac{d2^{4d}c_{\alpha}^{d}}{t_{j}^{d}(2^{d}-1)}
	\sum_{l=1}^{k}\left(I_{l}^{1,\alpha}+I_{l}^{2,\alpha}\right).
\end{align*}
Since the inequality
\begin{align*}
	\left|\left(1-\frac{r^{2}}{t_{j}^{2}}\right)^{\alpha-1}-1\right|=\left|\int_{0}^{r}\left(1-\frac{s^{2}}{t_{j}^{2}}\right)^{\alpha-2}(\alpha-1)\frac{-2s}{t_{j}^{2}}ds\right|\leq 3|\alpha-1|\cdot \frac{r^{2}}{t_{j}^{2}}
\end{align*}
holds for $r\leq t_{j}/(2^{3}c_{\alpha})$, then
\begin{align*}
	\sum_{l=1}^{k}\left|I_{l}^{1,\alpha}+I_{l}^{2,\alpha}\right|&\leq \frac{|\alpha-1|2^{d+12}c_{\alpha}^{2}}{t_{j}^{2}(d+2)}\cdot\frac{2^{8k(d+2)}c_{\alpha}^{2k(d+2)}-1}{2^{8(d+2)}c_{\alpha}^{2(d+2)}-1}.
\end{align*}
Hence,
\begin{align*}
	\left|A_{t_{j}}^{\alpha}f_{n}(0)\right|\leq \frac{1}{|\alpha-1|^{1+d}}\cdot\frac{1}{2^{4d+2}}\leq \frac{1}{2^{6}}.
\end{align*}

For $A_{t_{1}}^{\alpha}f_{n}(0)$, we have
\begin{align*}
	\left|A_{t_{1}}^{\alpha}f_{n}(0)-1\right|&=\frac{d}{2^{d}-1}\int_{1}^{2}\left|\left(1-\frac{r^{2}}{2^{8}c_{\alpha}^{2}}\right)^{\alpha-1}-1\right|r^{d-1}dr \\
	&\leq \frac{d}{2^{d}-1}\int_{1}^{2}\frac{3|\alpha-1|r^{d+1}}{2^{8}c_{\alpha}^{2}}dr\leq \frac{1}{8}.
\end{align*} 
Therefore,
\begin{align*}
	\left|A_{t_{1}}^{\alpha}f_{n}(0)\right|\geq \frac{7}{8}.
\end{align*}

If $3\leq j\leq n$ and $j=2k-1$ for some $k\in \mathbb{N}$, then
\begin{align*}
	A_{t_{j}}^{\alpha}f_{n}(0)&=\frac{\sigma(\mathbb{S}^{d-1})}{t_{j}^{d}\Gamma(\alpha)}\int_{0}^{t_{j}}\left(1-\frac{r^{2}}{t_{j}^{2}}\right)^{\alpha-1}\sum_{l=1}^{2k-1}\widetilde{c}_{l}\chi_{E_{l}}(r)r^{d-1}dr.
\end{align*}
By a similar argument as above, we obtain
\begin{align*}
	&\left|\frac{\sigma(\mathbb{S}^{d-1})}{t_{j}^{d}\Gamma(\alpha)}\int_{0}^{t_{j}}\left(1-\frac{r^{2}}{t_{j}^{2}}\right)^{\alpha-1}\sum_{l=1}^{2k-2}\widetilde{c}_{l}\chi_{E_{l}}(r)r^{d-1}dr\right| \\
	&\quad \leq \frac{d2^{4d}c_{\alpha}^{d}}{t_{j}^{d}(2^{d}-1)}\cdot \frac{|\alpha-1|2^{d+12}c_{\alpha}^{2}}{(d+2)t_{j}^{2}}\cdot\sum_{l=1}^{k-1}\left(2^{8l-8}c_{\alpha}^{2l-2}\right)^{d+2}\leq \frac{1}{2^{18}}.
\end{align*}
Observe that for $1\leq j\leq n$,
\begin{align*}
	&\left|\frac{\sigma(\mathbb{S}^{d-1})}{t_{j}^{d}\Gamma(\alpha)}\int_{0}^{t_{j}}\left(1-\frac{r^{2}}{t_{j}^{2}}\right)^{\alpha-1}\widetilde{c}_{j}\chi_{E_{j}}(r)r^{d-1}dr-1\right|\\
	&\quad \leq \frac{d2^{4d}c_{\alpha}^{d}}{2^{4jd}c_{\alpha}^{jd}(2^{d}-1)}\int_{2^{4j-4}c_{\alpha}^{j-1}}^{2^{4j-3}c_{\alpha}^{j-1}}\left|\left(1-\frac{r^{2}}{2^{8j}c_{\alpha}^{2j}}\right)^{\alpha-1}-1\right|r^{d-1}dr\\
	&\quad \leq \frac{d2^{4d}c_{\alpha}^{d}}{2^{4jd}c_{\alpha}^{jd}(2^{d}-1)}\cdot\frac{3|\alpha-1|}{2^{8j}c_{\alpha}^{2j}}\int_{2^{4j-4}c_{\alpha}^{j-1}}^{2^{4j-3}c_{\alpha}^{j-1}}r^{d+1}dr \leq \frac{1}{8}.
\end{align*}
Thus, 
\begin{align*}
	\left|\frac{\sigma(\mathbb{S}^{d-1})}{t_{j}^{d}\Gamma(\alpha)}\int_{0}^{t_{j}}\left(1-\frac{r^{2}}{t_{j}^{2}}\right)^{\alpha-1}\widetilde{c}_{j}\chi_{E_{j}}(r)r^{d-1}dr\right|\geq \frac{7}{8}.
\end{align*}
So, 
\begin{align*}
	\left|A_{t_{i}}^{\alpha}f_{n}(0)\right|\geq \frac{7}{8}-\frac{1}{2^{18}}
\end{align*}
holds when $1\leq j\leq n$ and $j$ is odd. 

Assume $|x|\leq 1/(2^{8+d}d)$ and $1\leq j\leq n$. Then
\begin{align*}
	\left|A_{t_{j}}^{\alpha}f_{n}(x)-A_{t_{j}}^{\alpha}f_{n}(0)\right|&\leq \frac{d2^{4d+1}c_{\alpha}^{d}}{t_{j}^{d}(2^{d}-1)\sigma(\mathbb{S}^{d-1})}\sum_{l=1}^{j}\int_{B(0,t_{j})}\left|\chi_{E_{l}}(|x-y|)-\chi_{E_{l}}(|y|)\right|dy \\
	&\leq \frac{d2^{4d+1}c_{\alpha}^{d}}{t_{j}^{d}(2^{d}-1)\sigma(\mathbb{S}^{d-1})}\sum_{l=1}^{j}m(E_{l,x}), 
\end{align*}
where
\begin{align*}
	E_{l,x}
	:=&\left\{y:2^{4(l-1)}c_{\alpha}^{l-1}-|x|\leq |y|\leq 2^{4(l-1)}c_{\alpha}^{l-1}+|x|\right\}\\
	&\quad\bigcup \left\{y:2^{4(l-1)+1}c_{\alpha}^{l-1}-|x|\leq |y|\leq 2^{4(l-1)+1}c_{\alpha}^{l-1}+|x|\right\}.
\end{align*}
Note that
\begin{align*}
	\sum_{l=1}^{j}m(E_{l,x})&\leq \sum_{l=1}^{j}\frac{\sigma(\mathbb{S}^{d-1})}{d}\left[\left(2^{4(l-1)+1}c_{\alpha}^{l-1}+|x|\right)^{d}-\left(2^{4(l-1)+1}c_{\alpha}^{l-1}-|x|\right)^{d}\right.\\
	&\quad \left. +\left(2^{4(l-1)}c_{\alpha}^{l-1}+|x|\right)^{d}-\left(2^{4(l-1)}c_{\alpha}^{l-1}-|x|\right)^{d}\right]\\
	&\leq\sigma(\mathbb{S}^{d-1})2^{2d}|x|\frac{2^{4(d-1)j}c_{\alpha}^{(d-1)j}-1}{2^{4(d-1)}c_{\alpha}^{(d-1)}-1}.
\end{align*}
Hence, 
\begin{align*}
	\left|A_{t_{j}}^{\alpha}f_{n}(x)-A_{t_{j}}^{\alpha}f_{n}(0)\right|\leq \frac{1}{2^{5}},
\end{align*}
from which it follows that
\begin{align*}
	\left|A_{t_{j}}^{\alpha}f_{n}(x)\right|\leq \frac{1}{2^{4}}
\end{align*}
when $j$ is even and
\begin{align*}
	\left|A_{t_{j}}^{\alpha}f_{n}(x)\right|\geq \frac{7}{8}-\frac{1}{2^{18}}-\frac{1}{2^{5}}
\end{align*}
when $j$ is odd. Therefore,
\begin{align*}
	\left(\sum_{j=1}^{n-1}\left|A_{t_{j+1}}^{1}f_{n}(x)-A_{t_{j}}^{1}f_{n}(x)\right|^{q}\right)^{1/q}\geq \frac{(n-1)^{1/q}}{4}.
\end{align*}
Combining this with the fact
\begin{align*}
	\|f_{n}\|_{L^{\infty}(\mr^{d})}\leq \frac{d2^{4d}c_{\alpha}^{d}|\Gamma(\alpha)|}{(2^{d}-1)\sigma(\mathbb{S}^{d-1})},
\end{align*}
we conclude that $V_{q}(\mathcal{A}^{\alpha})$ is unbounded on $L^{\infty}(\mathbb{R}^{d})$ when $|\alpha-1|\geq 1$.

The remaining case $|\alpha-1|<1$ follows from taking $c_{\alpha}:=1$ in the above argument.
\end{proof}
\bigskip
\noindent {\bf Acknowledgments:}  Wenjuan Li is supported by the NNSF of China
(Grant No. 12271435). Dongyong Yang is supported by the NNSF of China (Grant No. 12171399).

\begin{thebibliography}{10}
	
	\bibitem{Beltran21localsmooth}
	D.~Beltran, J.~Hickman, and C.~D. Sogge.
	\newblock Sharp local smoothing estimates for {F}ourier integral operators.
	\newblock In {\em Geometric aspects of harmonic analysis}, Springer INdAM Ser,
	vol. 45, pages 29--105. Springer, Cham, 2021.
	
	\bibitem{Beltran2022MathAnn}
	D.~Beltran, R.~Oberlin, L.~Roncal, A.~Seeger, and B.~Stovall.
	\newblock Variation bounds for spherical averages.
	\newblock {\em Math. Ann.}, 382(1-2):459--512, 2022.
	
	\bibitem{beltran2022multiscale}
	D.~Beltran, J.~Roos, and A.~Seeger.
	\newblock Multi-scale sparse domination.
	\newblock {\em arXiv:2009.00227}, 2022.
	
	\bibitem{Bou86JAM}
	J.~Bourgain.
	\newblock Averages in the plane over convex curves and maximal operators.
	\newblock {\em J. Analyse Math.}, 47:69--85, 1986.
	
	\bibitem{MR1019960}
	J.~Bourgain.
	\newblock Pointwise ergodic theorems for arithmetic sets.
	\newblock {\em Inst. Hautes \'{E}tudes Sci. Publ. Math}, (69):5--45, 1989.
	
	\bibitem{Bou15AnnMath}
	J.~Bourgain and C.~Demeter.
	\newblock The proof of the {$l^2$} decoupling conjecture.
	\newblock {\em Ann. of Math. (2)}, 182(1):351--389, 2015.
	
	\bibitem{Javier1986Invent}
	J.~Duoandikoetxea and J.~L. Rubio~de Francia.
	\newblock Maximal and singular integral operators via {F}ourier transform
	estimates.
	\newblock {\em Invent. Math.}, 84(3):541--561, 1986.
	
	\bibitem{Friz20AnnProb}
	P.~K. Friz and P.~Zorin-Kranich.
	\newblock Rough semimartingales and {$p$}-variation estimates for martingale
	transforms.
	\newblock {\em Ann. Probab.}, 51(2):397--441, 2023.
	
	\bibitem{Guo2020Anal}
	S.~Guo, J.~Roos, and P.-L. Yung.
	\newblock Sharp variation-norm estimates for oscillatory integrals related to
	{C}arleson's theorem.
	\newblock {\em Anal. PDE}, 13(5):1457--1500, 2020.
	
	\bibitem{Guth2020Ann}
	L.~Guth, H.~Wang, and R.~Zhang.
	\newblock A sharp square function estimate for the cone in {$\Bbb {R}^3$}.
	\newblock {\em Ann. of Math. (2)}, 192(2):551--581, 2020.
	
	\bibitem{Jones03Israel}
	R.~L. Jones, J.~M. Rosenblatt, and M.~Wierdl.
	\newblock Oscillation in ergodic theory: higher dimensional results.
	\newblock {\em Israel J. Math.}, 135:1--27, 2003.
	
	\bibitem{Jones08Trans}
	R.~L. Jones, A.~Seeger, and J.~Wright.
	\newblock Strong variational and jump inequalities in harmonic analysis.
	\newblock {\em Trans. Amer. Math. Soc.}, 360(12):6711--6742, 2008.
	
	\bibitem{Jones04TransAMS}
	R.~L. Jones and G.~Wang.
	\newblock Variation inequalities for the {F}ej\'{e}r and {P}oisson kernels.
	\newblock {\em Trans. Amer. Math. Soc.}, 356(11):4493--4518, 2004.
	
	\bibitem{Krause2018Ergodic}
	B.~Krause and P.~Zorin-Kranich.
	\newblock Weighted and vector-valued variational estimates for ergodic
	averages.
	\newblock {\em Ergodic Theory Dynam. Systems}, 38(1):244--256, 2018.
	
	\bibitem{MR420837}
	D.~L\'{e}pingle.
	\newblock La variation d'ordre {$p$} des semi-martingales.
	\newblock {\em Z. Wahrscheinlichkeitstheorie und Verw. Gebiete},
	36(4):295--316, 1976.
	
	\bibitem{Liu2023arXiv}
	N.~Liu, M.~Shen, L.~Song, and L.~Yan.
	\newblock {$L^p$} bounds for {S}tein's spherical maximal operators.
	\newblock {\em arXiv:2303.08655}, 2023.
	
	\bibitem{MR3671712}
	T.~Ma, J.~L. Torrea, and Q.~Xu.
	\newblock Weighted variation inequalities for differential operators and
	singular integrals in higher dimensions.
	\newblock {\em Sci. China Math}, 60(8):1419--1442, 2017.
	
	\bibitem{miao2017Proc}
	C.~Miao, J.~Yang, and J.~Zheng.
	\newblock On local smoothing problems and {S}tein's maximal spherical means.
	\newblock {\em Proc. Amer. Math. Soc.}, 145(10):4269--4282, 2017.
	
	\bibitem{Mir17Inv}
	M.~Mirek, E.~M. Stein, and B.~Trojan.
	\newblock {$\ell^p(\Bbb Z^d) $}-estimates for discrete operators of {R}adon
	type: variational estimates.
	\newblock {\em Invent. Math.}, 209(3):665--748, 2017.
	
	\bibitem{Mirek20AnalPDE}
	M.~Mirek, E.~M. Stein, and P.~Zorin-Kranich.
	\newblock A bootstrapping approach to jump inequalities and their applications.
	\newblock {\em Anal. PDE}, 13(2):527--558, 2020.
	
	\bibitem{Mir20Adv}
	M.~Mirek, E.~M. Stein, and P.~Zorin-Kranich.
	\newblock Jump inequalities for translation-invariant operators of {R}adon type
	on {$\Bbb Z^d$}.
	\newblock {\em Adv. Math.}, 365:107065, 57, 2020.
	
	\bibitem{Mir20MathAnn}
	M.~Mirek, E.~M. Stein, and P.~Zorin-Kranich.
	\newblock Jump inequalities via real interpolation.
	\newblock {\em Math. Ann.}, 376(1-2):797--819, 2020.
	
	\bibitem{Mockenhaupt1992Ann}
	G.~Mockenhaupt, A.~Seeger, and C.~D. Sogge.
	\newblock Wave front sets, local smoothing and {B}ourgain's circular maximal
	theorem.
	\newblock {\em Ann. of Math. (2)}, 136(1):207--218, 1992.
	
	\bibitem{Nowak23CPAA}
	A.~Nowak, L.~Roncal, and T.~Z. Szarek.
	\newblock Endpoint estimates and optimality for the generalized spherical
	maximal operator on radial functions.
	\newblock {\em Commun. Pure Appl. Anal.}, 22(7):2233--2277, 2023.
	
	\bibitem{Obe12JEMS}
	R.~Oberlin, A.~Seeger, T.~Tao, C.~Thiele, and J.~Wright.
	\newblock A variation norm {C}arleson theorem.
	\newblock {\em J. Eur. Math. Soc. (JEMS)}, 14(2):421--464, 2012.
	
	\bibitem{Pisier88Prob}
	G.~Pisier and Q.~H. Xu.
	\newblock The strong {$p$}-variation of martingales and orthogonal series.
	\newblock {\em Probab. Theory Related Fields}, 77(4):497--514, 1988.
	
	\bibitem{Qian1998AnnProb}
	J.~Qian.
	\newblock The {$p$}-variation of partial sum processes and the empirical
	process.
	\newblock {\em Ann. Probab.}, 26(3):1370--1383, 1998.
	
	\bibitem{Rubio1986Duke}
	J.~Rubio~de Francia.
	\newblock Maximal functions and {F}ourier transforms.
	\newblock {\em Duke Math. J.}, 53(2):395--404, 1986.
	
	\bibitem{Stein1976Acad}
	E.~M. Stein.
	\newblock Maximal functions. {I}. {S}pherical means.
	\newblock {\em Proc. Nat. Acad. Sci. U.S.A.}, 73(7):2174--2175, 1976.
	
	\bibitem{Stein93}
	E.~M. Stein.
	\newblock {\em Harmonic Analysis: Real-variable Methods, Orthogonality, and
		Oscillatory Integrals}.
	\newblock Princeton Mathematical Series, vol. 43. Princeton University Press,
	Princeton, NJ, 1993.
	
	\bibitem{Stein1973book}
	E.~M. Stein and G.~Weiss.
	\newblock {\em Introduction to {F}ourier Analysis on {E}uclidean Spaces}.
	\newblock Princeton Mathematical Series, No. 32. Princeton University Press,
	Princeton, NJ, 1971.
	
	\bibitem{bookWatson}
	G.~N. Watson.
	\newblock {\em A Treatise on the Theory of {B}essel Functions}.
	\newblock Cambridge Mathematical Library. Cambridge University Press,
	Cambridge, 1995.
	
\end{thebibliography}

\end{document}